\documentclass[10pt,leqno]{article}
\usepackage[cm]{fullpage}
\usepackage{enumitem}
\usepackage{amscd}
\usepackage{ amssymb }
\usepackage{amsmath}
\usepackage{ esint }
\usepackage{systeme,mathtools}
\usepackage{amsfonts}
\usepackage{orcidlink}
\definecolor{orcidlogocol}{HTML}{A6CE39}
\usepackage{amsthm}
\usepackage{tikz-cd}
\usepackage{stmaryrd}
\usepackage{lipsum}
\usepackage{appendix}
\usepackage{hyperref}
\usepackage{cleveref}
\DeclareMathOperator{\divg}{div}

\DeclareMathOperator{\tr}{tr}

\DeclareMathOperator{\Rc}{Rc}

\begin{document}

\theoremstyle{definition}
\newtheorem{claim}{Claim}
\theoremstyle{plain}
\newtheorem{proposition}{Proposition}[section]
\newtheorem{theorem}[proposition]{Theorem}
\newtheorem{lemma}[proposition]{Lemma}
\newtheorem{corollary}[proposition]{Corollary}
\theoremstyle{definition}
\newtheorem{defn}[proposition]{Definition}
\theoremstyle{remark}
\newtheorem{remark}[proposition]{Remark}
\theoremstyle{definition}
\newtheorem{example}[proposition]{Example}
\theoremstyle{definition}
\newtheorem*{Motivation}{Motivation}
\newcommand{\RomanNumeralCaps}[1]
    {\MakeUppercase{\romannumeral #1}}
\newcommand{\Addresses}{{
  \footnotesize
 
  \par\nopagebreak
  \textsc{Department of Mathematics, University of California-Irvine, CA, U.S.A }\par\nopagebreak
  \textit{E-mail address}: \texttt{kuanhuil@uci.edu}
  \orcidlink{0000-0001-5109-4797}
}}

\title{\textbf{The linear stability of non-Kähler Calabi-Yau metrics}}
\author{KUAN-HUI LEE}
\date{}
\maketitle
\Addresses
\begin{abstract}
Non-Kähler Calabi-Yau theory is a newly developed subject and it arises naturally in mathematical physics and generalized geometry. The relevant geometries are pluriclosed metrics which are critical points of the generalized Einstein--Hilbert action which is an extension of Perelman’s $\mathcal{F}$-functional. In this work, we study the critical points of the generalized Einstein-Hilbert action and discuss the stability of critical points which are defined as pluriclosed steady solitons. We proved that all compact Bismut--Hermitian--Einstein metrics are linearly stable which is non-Kähler analogue of the stability results of Ricci solitons from Tian, Zhu \cite{tian2018perelmans} and Hall, Murphy \cite{H1}, Koiso \cite{Ko}. In addition, all compact Bismut-flat pluriclosed metrics with $(2n-1)$-positive Ricci curvature are strictly linearly stable when the complex structure is fixed.
\end{abstract}

\section{Introduction}

In complex geometry, the Calabi-Yau manifolds are defined to be compact Kähler manifolds $M$ with vanishing first Chern class. Due to the Calabi-Yau Theorem \cite{Yau} these are precisely the compact manifolds that admit Ricci-flat Kähler metrics. Later, Cao \cite{Cao1985} used the Kähler-Ricci flow to prove the Calabi-Yau theorem which shows global existence and convergence to a Ricci-flat metric with arbitrary initial data. In recent years, these profound results have been extended to a more general settings in complex geometry. Given a Hermitian manifold $(M,g,J)$ with the Kähler form $\omega$, we say that the metric is pluriclosed if 
\begin{align*}
    dd^c\omega=0.
\end{align*}
This is a natural generalization of the Kähler condition $d\omega=0$. In fact, pluriclosed metrics exist on all compact complex surfaces \cite{Gauduchon1984}. The pluriclosed metrics are studied in the theory of non-Kähler manifolds and in the physics of supersymmetry. In the following, we focus on the Bismut connection (\cite{Bismut1989})
\begin{align*}
    \nabla^B=\nabla+\frac{1}{2}g^{-1}H,
\end{align*}
where $H=-d^c\omega$ and $\nabla$ denotes the Levi-Civita connection. To generalize the Ricci form, we then define 
\begin{align*}
    \rho_B=\frac{1}{2}\text{tr  }R^BJ \in \Lambda^2(M),
\end{align*}
which determines a representative of $2\pi c_1(M)$. We say that a pluriclosed metric is \emph{Bismut Hermitian--Einstein } if 
\begin{align*}
    \rho_B=0.
\end{align*}
In fact, as shown in \cite{GRF} Proposition 8.10,
\begin{align*}
    \rho_B^{1,1}(\cdot,J\cdot)&=\Rc-\frac{1}{4}H^2+L_{\theta^\sharp}g,
   \\\rho_B^{2,0+0,2}(\cdot,J\cdot)&=-\frac{1}{2}d^*H+\frac{1}{2}d^B\theta,
\end{align*}
where $\theta$ is the Lee form given by $\theta=-d^*\omega\circ J$, $d^B$ is the exterior derivative with respect to the Bismut connection $\nabla^B$ and $H^2(X,Y)=\langle \iota_XH,\iota_YH \rangle$. One may view the equation $\rho_B=0$ as a generalization of the Calabi-Yau condition to non-Kähler, pluriclosed manifolds.

Next, we consider the Einstein--Hilbert functional (also known as the total scalar curvature functional, see \cite{B} for detail)
\begin{align*}
    S(g)=\int_M R(g)dV_g.
\end{align*}
It is well known that the Einstein metrics are critical points of the Einstein-Hilbert functional restricted to the
space of Riemannian metrics with total volume 1. (\cite{Ko}). Define the generalized Einstein--Hilbert functional $\mathcal{F}$, a generalization of Perelman's functional \cite{Pe}, which is given by 
\begin{align*}
    \mathcal{F}\colon\nonumber&\quad \Gamma(S^2M)\times \Omega^2\times C^\infty (M)\to \mathbb{R}
    \\& \quad (g,b,f)\longmapsto \int_M (R-\frac{1}{12}|H_0+db|^2+|\nabla f|^2)e^{-f}dV_g,  
\end{align*}
and write
\begin{align*}
    \lambda(g,b)\coloneqq\inf\Big\{\mathcal{F}(g,b,f)\big|\kern0.2em f\in C^\infty(M),\,\int_M e^{-f}dV_g=1\Big\},
\end{align*}
where $H_0$ is a fixed closed 3-form (\cite{P}). Critical points of the generalized Einstein--Hilbert functional are steady generalized Ricci solitons (c.f Definition \ref{soliton}) and they satisfy 
\begin{align}
    0=\Rc-\frac{1}{4}H^2+\nabla^2 f, \quad 0=d_g^*H+i_{\nabla f}H \label{RRR}.
\end{align}
These match the equations of motion in compactifications of type \RomanNumeralCaps{2} supergravity \cite{Callan:1985ia}. In the following, we focus on pluriclosed steady solitons which are pluriclosed metrics satisfying the generalized Ricci soliton equations (\ref{RRR}). Furthermore, one can also see that any compact Bismut Hermitian--Einstein metric must be a pluriclosed steady soliton (See \cite{MJJ} Proposition 2.6). However, the converse statement is false (c.f \Cref{E1}). It is interesting to ask whether Bismut Hermitian--Einstein metrics and generalized Einstein metrics are equivalent or not (c.f Definition \ref{soliton}). Until now, we still do not know the answer to the problem, but in the 4-dimensional case, we can prove that Bismut-Hermitian-Einstein metrics are definitely generalized Einstein metrics. In fact, they are not only generalized Einstein metrics, but they are also Bismut-flat (See \cite{GRF} Theorem 8.26).

In this work, we discuss the stability of pluriclosed steady solitons. As shown in \cite{H2}, Hall and Murphy proved that a Kähler-Ricci soliton with $\text{dim } H^{(1,1)}(M)\geq 2$ is linearly unstable. However, in \cite{tian2018perelmans}, Tian and Zhu showed that the second variation of Perelman's entropy on Kähler-Einstein manifold $(M,g)$ is non-positive in the space of Kähler metrics with Kähler classes cohomologous to $2\pi c_1(M)$ (Notice that complex structures on $M$ may vary). We then see that Kähler Ricci solitons may not be linearly stable in general but they are stable under some situations. Later, Hall and Murphy generalized Tian and Zhu's result and conjectured that the second variation of Perelman's entropy on shrinking Kähler-Ricci soliton is non-positive in \cite{H1}.

In our case, let $(M,g,H,J,f)$ be a pluriclosed steady soliton with a closed 3-form $H=-d^c\omega$. Consider a family $(M,g_t,H_t,J_t,f_t)$ with 
\begin{align*}
    H_t=H_0+db_t,\quad (M,g_0,H_0,J_0,b_0)=(M,g,H,J,0)
\end{align*}
and $f_t$ is the minimizer of $\lambda(g_t,b_t)$. Denote 
\begin{align*}
   \frac{\partial }{\partial t}\big|_{t=0} g=h, \quad  \frac{\partial }{\partial t}\big|_{t=0} b=\beta,\quad   \frac{\partial }{\partial t}\big|_{t=0}\omega =\phi, \quad  \frac{\partial }{\partial t}\big|_{t=0} J =I,
\end{align*}
where $I$ is an infinitesimal variation of complex structure. Since $\lambda$ is diffeomorphism invariant, we may take our general variation  $\gamma=h-\beta\in\ker\overline{\divg}_f$ (See (\ref{bar1}) for the definition of $\overline{\divg}$). Our goal is to study the general variation $\gamma$ in the pluriclosed setting. Incorporating information on the deformation of complex structure \cite{B,Ko}, we are able to precisely analyze $\gamma$. More precisely, we derive the following proposition.

\begin{proposition}\label{P1}
Let $(M,g,H,J,f)$ be a pluriclosed steady soliton with a closed 3-form $H=-d^c\omega$.  For any $\gamma\in\ker\overline{\divg}_f$, there exists a 2-form $\xi$ with $(d^B_f)^*\xi=0$ and a symmetric 2-tensor $\eta$ with $\divg_f^B\eta=0$, satisfying (\ref{k2}) such that
\begin{align}
    \gamma(X,Y)=\xi(X,JY)+\eta(X,Y). \label{4}
\end{align}
where $\divg_f^B$ is the divergence operator and $(d^B_f)^*$ is the formal adjoint of $d^B$ with respect to the Bismut connection (\ref{20}). Moreover, the pair $(\xi,\eta)$ is unique.
\end{proposition}

Based on the proper decomposition of variation, we further compute the second variation at $\gamma$. Recall that in \cite{KK} Theorem 1.1, the second variation is given by the operator $\overline{\mathcal{L}}_f(\gamma)=\frac{1}{2}\overline{\triangle}_f\gamma+\mathring{R}^+(\gamma)$ where $\overline{\triangle}_f$ 
 is defined in (\ref{bar3}) and $R^+$ is the Bismut curvature given in \Cref{P3}. For $\xi$-part, we reduce $\overline{\triangle}$ to the Hodge Laplacian $\triangle^B_{d}$ with respect to the Bismut connection. For $\eta$-part, we apply the properties of infinitesimal complex structure to simplify the operator $\overline{\mathcal{L}}_f$. Our result is given by
\begin{theorem}\label{T2}

Let $(M,g,H,J,f)$ be a compact pluriclosed steady soliton and $\gamma\in\ker\overline{\divg}_f$. Then, 
\begin{align}
     \frac{d^2}{dt^2}\Big|_{t=0}\lambda(\gamma)&= -2\|d^*_f\xi\|_f^2-\frac{1}{6}\|d\xi-C\|_f^2 +\int_M |\eta|^2\big( -\frac{1}{2}S_B+\mathcal{L}_Vf \big)e^{-f}dV_g, \label{k4}
\end{align} 
where $\gamma=\xi\circ J+\eta$ is the decomposition given in (\ref{k3}), $C$ is defined in (\ref{C}), $\|\cdot\|_f$ denotes the norm of $f$-twisted $L^2$ inner product (\ref{6}),  $S_B=\tr_\omega \rho_B$ denotes the Bismut scalar curvature and the vector field $V=\frac{1}{2}(\theta^\sharp-\nabla f)$. In particular, when $(M,g,H,J,f)$ is a Bismut--Hermitian--Einstein manifold, 
\begin{align*}
     \frac{d^2}{dt^2}\Big|_{t=0}\lambda(\gamma)&= -2\|d^*_f\xi\|_f^2-\frac{1}{6}\|d\xi-C\|_f^2\leq 0.
\end{align*}

\end{theorem}

As an application, we obtain stability result for \emph{generalized Ricci flow} which is defined as a solution of a one parameter pair of Riemannian metrics and 2-forms $(g_t,b_t)$ satisfying 
\begin{align*}
    \frac{\partial}{\partial t}g=-2\Rc+\frac{1}{2}H^2, \quad
   \frac{\partial}{\partial t}b=-d^*H,  
\end{align*}
where $H=H_0+db$ and $H_0$ is a fixed closed 3-form. Notice that the generalized Ricci flow is closely related to the mathematical physics \cite{polchinski_1998} and generalized geometry \cite{MG,hitchin}. In the complex case, the generalized Ricci flow is gauge equivalent to pluriclosed flow \cite{J8}. This geometric flow is first introduced by Streets and Tian \cite{J2,J3}. More precisely, a one parameter family of pluriclosed metrics and 2-forms $(g_t,\beta_t)$ is called a solution of \emph{pluriclosed flow} if it satisfies
\begin{align*}
        \frac{\partial}{\partial t}\omega=-\rho^{1,1}_B,\quad  \frac{\partial}{\partial t}\beta=-\rho^{2,0}_B. 
    \end{align*}
By now there are many global existence and convergence results for pluriclosed flow (\cite{J9,MJJ,J6}). We note that fixed points of pluriclosed flow are Bismut--Hermitian--Einstein manifolds. As a matter of fact, Bismut-flat manifolds are the only known examples of non-Kähler Bismut--Hermitian--Einstein to the present day. As proved in \cite{MJJ} Theorem 1.2, given a pluriclosed metric $\omega_0$ with $[\partial\omega_0]=[\partial\omega_F]$ where $(M,\omega_F,J)$ is a compact Bismut-flat manifold, the solution to pluriclosed flow with initial data $\omega_0$ exists for all time and converges to a Bismut-flat metric. We expect that all Bismut-flat metrics should be linearly stable and \cite{KK} Corollary 1.2 justified this result.  In \cite{K}, the author proved that the dynamically stable and linearly stable are equivalent (c.f \cite{K} Theorem 1.8 in Riemannian manifold case) if all infinitesimal deformation is integrable (c.f Definition \ref{IS}). Thus, we prove that 

\begin{corollary}
Let $(M,g,H,J)$ be a compact Bismut--Hermitian--Einstein metric. If all infinitesimal deformation is integrable, then $(M,g,H,J)$ is dynamically stable. In other words, for any neighborhood of $(g,0)\in \mathcal{M}\times \Omega^2$, where $\mathcal{M}$ denotes the space of Riemannian metrics,  there exists a smaller neighborhood $\mathcal{V}\subset\mathcal{U}$ such that the pluriclosed flow $(g_t,b_t)$ starting in $\mathcal{V}$ stays in $\mathcal{U}$ for all $t\geq 0$ and converges to a critical point up to a diffeomorphism.          
\end{corollary}
  
In \cite{KK}, the author also discussed infinitesimal solitonic deformations of generalized Ricci solitons. In the Bismut flat manifolds case, infinitesimal solitonic deformations could be understood as the kernel of second variation. We then prove that    

\begin{proposition}
Let $(M^{2n},g,H,J)$ be a compact Bismut-flat, pluriclosed manifold. Suppose that the Ricci curvature is $(2n-1)$-positive, i.e., there is at most one non-positive eigenvalue then there is no essential infinitesimal generalized solitonic deformation on directions of fixing the complex structure.        
\end{proposition}

As examples, we recall that the Hopf surface is a Bismut-flat manifold that is diffeomorphic to $S^1\times S^3$ and the Calabi-Eckmann space is also a Bismut-flat manifold that is diffeomorphic to $S^3\times S^3$ (See \Cref{E1} and \Cref{E2} for details). Besides, compact Bismut-flat
Hermitian manifolds are classified by Q. Wang, B. Yang and F. Zheng in \cite{BF}. We have the following corollary.
\begin{corollary}
The Hopf surface and Calabi-Eckmann space with the Bismut flat pluriclosed metrics described in \Cref{E1} and \Cref{E2} respectively are strictly linearly stable on the directions fixing the complex structure.     
\end{corollary}

Here is an outline of the rest of the paper. In section 2, we discuss the non-Kähler geometry, including Bismut connection, generalized Einstein--Hilbert functional. Besides, we introduce a special complex coordinate that we will use in later computation. In section 3, we first study the infinitesimal complex structure and studied the decomposition of variation space. Then, we prove the main result by discussing anti-symmetric part and symmetric part respectively. In the last part, we focus on Bismut-flat manifolds with $(2n-1)$ positive Ricci curvature and show that there is no infinitesimal deformation on the direction of fixing the complex structure.  

\vspace{5pt}
\textbf{Acknowledgements:} The author is grateful to his advisor Jeffrey D.~Streets and Mario Garcia--Fernandez for their helpful advices. Their suggestions play an important role in this work. 
\vspace{5pt}

\section{Non-Kähler geometry}

\subsection{Notation}
In this work, the convention of Riemann curvature is given by 
\begin{align*}
    R(X,Y,Z,W)=\langle \nabla_X\nabla_YZ-\nabla_Y\nabla_XZ-\nabla_{[X,Y]}Z,W \rangle, \quad R_{ijkl}=R(e_i,e_j,e_k,e_l).
\end{align*}
We adopt the following notations.

\begin{itemize}
    \item For any 2-tensor $\gamma$,
\begin{align*}
    \mathring{R} (\gamma)_{ij}= R_{iklj}\gamma_{kl}. 
\end{align*}
    \item The $f$-twisted $L^2$ inner product is given by
\begin{align}
    \Big( \gamma_1,\gamma_2\Big)_{f}\coloneqq\int_M \langle \gamma_1,\gamma_2 \rangle_g  e^{-f} dV_g, \label{6}
\end{align}
where $\langle\kern0.3em,\kern0.3em\rangle_g$ denotes the standard inner product induced by a Riemannian metric $g$ and $\gamma_1,\gamma_2$ are tensors of same type. In particular, we mainly focus on the case when $\gamma$ is a 2-tensor.
\item The $f$-twisted Laplacian $\triangle_f$ is given by
\begin{align*}
    \triangle_f=\triangle-\nabla f\cdot \nabla. 
\end{align*}
 \item Let $S^pM$ denote the bundle of symmetric $(0,p)-$tensor. The $f$-twisted divergence operator $\divg_f: C^\infty(S^pM)\rightarrow  C^\infty(S^{p-1}M)$ is given by 
  \begin{align*}
     (\divg_f T)(X_1,...,X_{p-1})=-\sum_{i=1}^n (\nabla_{e_i}T)(e_i,X_1,...,X_{p-1})+T(\nabla f,X_1,...,X_{p-1}).
 \end{align*}
 \item On a complex manifold $(M,J)$, we denote
\begin{align*}
    (\gamma\circ J)(\cdot,\cdot)=\gamma(\cdot,J\cdot),
\end{align*}
for any 2-tensor $\gamma$. In addition, in the following we will use the index $i,j,k,...$ to represent real coordinates and the index $\alpha,\beta,\gamma,...$ to represent complex coordinates. 
\end{itemize}

\subsection{Bismut Connections and curvatures}

In this subsection, let us briefly review the definition of Bismut connections. Most details could be found in \cite{GRF}, \cite{KK}.
\begin{defn}
Let $(M,g,H)$ be a Riemannian manifold and $H\in\Omega^3$. 
\begin{itemize}
    \item The \emph{Bismut connections} $\nabla^{\pm}$ associated to $(g,H)$ are defined as 
\begin{align}
    \langle \nabla_X^{\pm}Y,Z \rangle=\langle \nabla_XY,Z \rangle\pm \frac{1}{2}H(X,Y,Z). 
\end{align}
Here, $\nabla$ is the Levi-Civita connection associated with $g$, i.e., $ \nabla^{\pm}$ are the unique compatible connections with torsion $\pm H$.
\item The \emph{mixed Bismut connection} is a connection $\overline{\nabla}$ on 2-tensors is defined as
\begin{align}
    (\overline{\nabla}_X\gamma)(Y,Z)=\nabla_X\Big(\gamma(Y,Z)\Big)-\gamma(\nabla^-_XY,Z)-\gamma(Y,\nabla^+_XZ). \label{MBC}
\end{align}
\item The \emph{$f$-twisted divergence operator} $\overline{\divg}_f:  \otimes^2T^*M\longrightarrow T^*M\times T^*M$ on 2-tensors with respect to $\overline{\nabla}$ is given by
\begin{align}
       \overline{\divg}_f\gamma&=(u,v), \label{bar1}
\end{align}
where $u_l=(\nabla^+)^m\gamma_{ml}-\nabla_mf \gamma_{ml}$ and $v_l=(\nabla^-)^m\gamma_{lm}-\nabla_mf \gamma_{lm}$. 
\item The \emph{$f$-twisted divergence operator} $\overline{\divg}_f:  T^*M\times T^*M\longrightarrow C^\infty(M)$ on $T^*M\times T^*M$ with respect to $\overline{\nabla}$ is given by
\begin{align}
       \overline{\divg}_f(u,v)=\frac{1}{2}(\divg_f u+\divg_f v). \label{bar2} 
\end{align}
Note that (\ref{bar2}) is the divergence operator acting at $T^*M\times T^*M$ is different from (\ref{bar1}).
\item The \emph{Laplace operator of the mixed Bismut connection} $\overline{\triangle}_f$ is defined by 
\begin{align}
    \overline{\triangle}_f=-\overline{\nabla}^{*_f}\overline{\nabla}\label{bar3} 
\end{align}
where $\overline{\nabla}^{*_f}$ is the formal adjoint of $\overline{\nabla}$ with respect to $f$-twisted $L^2$ inner product (\ref{6}). More precisely, suppose $\gamma$ is a 2-tensor,
\begin{align}
\overline{\triangle}_f\gamma_{ij}=\triangle_f\gamma_{ij}-H_{mjk}\nabla_m\gamma_{ik}+H_{mik}\nabla_m\gamma_{kj}-\frac{1}{4}(H_{jl}^2\gamma_{il}+H_{il}^2\gamma_{lj})-\frac{1}{2}H_{mkj}H_{mli}\gamma_{lk}. \label{bar4}  
\end{align}
\end{itemize}

\end{defn}

\begin{remark} \label{RRRRR}
    We remark that in the later computation, we sometime express the variation $\gamma=h-K$, where $h$ denotes the symmetric part and $-K$ denotes the skew-symmetric part. Then, 
    \begin{align*}
        \overline{\divg}_f \gamma=(u,v),
    \end{align*}
    where $u_l=(\divg_f h)_l-\frac{1}{2}H_{mkl}K_{mk}+(d_f^*K)_l$ and $v_l=(\divg_f h)_l-\frac{1}{2}H_{mkl}K_{mk}-(d_f^*K)_l$. Besides, 
    \begin{align*}
        \overline{\divg}_f\overline{\divg}_f\gamma=\divg_f\divg_f h-\frac{1}{6}\langle dK,H \rangle.
    \end{align*}
\end{remark}

Following the definitions, we are able to compute the curvature tensor of the Bismut connection and the Bianchi identity of Bismut curvature.
\begin{proposition}[\cite{GRF} Proposition 3.18]\label{P3}
Let $(M^n,g,H)$ be a Riemannian manifold with $H\in \Omega^3$ and $dH=0$, then for any vector fields $X,Y,Z,W$ we have
\begin{align*}
    \nonumber Rm^+(X,Y,Z,W)=& \kern0.2em Rm(X,Y,Z,W)+\frac{1}{2}\nabla_XH(Y,Z,W)-\frac{1}{2}\nabla_YH(X,Z,W)
    \nonumber\\&-\frac{1}{4}\langle H(X,W),H(Y,Z)\rangle+\frac{1}{4}\langle H(Y,W),H(X,Z)\rangle, 
    \\\Rc^+&=\Rc-\frac{1}{4}H^2-\frac{1}{2}d^*H, \quad R^+=R-\frac{1}{4}|H|^2, 
\end{align*}
where $H^2(X,Y)=\langle i_XH,i_YH\rangle$. Here $Rm^+$, $\Rc^+$, $R^+$ denote the Riemannian curvature, Ricci curvature, and scalar curvature with respect to the Bismut connection $\nabla^+$. In particular, if $(M,g,H)$ is Bismut-flat, then 
\begin{align}
    Rm(X,Y,Z,W)=\frac{1}{4}\langle H(X,W),H(Y,Z)\rangle-\frac{1}{4}\langle H(Y,W),H(X,Z)\rangle\quad \text{and } \quad \nabla H=0 \label{R}
\end{align}
for any vector fields $X,Y,Z,W$.
\end{proposition}
\begin{proposition}[\cite{GRF} Proposition 3.20]
Let $(M^n,g,H)$ be a Riemannian manifold with $H\in \Omega^3$ and $dH=0$, then for any vector fields $X,Y,Z,W$ we have
\begin{align}
    \nonumber\sum_{\sigma(X,Y,Z)}R^+(X,Y,Z,W)&=(\nabla^+_WH)(X,Y,Z)- \sum_{\sigma(X,Y,Z)}g\big(H(X,Y),H(Z,W)\big)
    \\&=\frac{1}{2}\Big( \sum_{\sigma(X,Y,Z)}(\nabla^+_XH)(Y,Z,W)+(\nabla^+_WH)(X,Y,Z) \Big).\label{10}
\end{align}
\end{proposition}
In terms of the Bismut connection $\nabla^+$, we have the following commutation formula.
\begin{lemma}
    Let $\gamma$ be a 2-tensor and $\{e_i\}$ be an orthonormal frame at a fixed point $p$. Then,
    \begin{align}
        (\nabla^+_i\nabla^+_j\gamma)_{kl}-(\nabla^+_j\nabla^+_i\gamma)_{kl}=-H_{ijm}(\nabla^+_m\gamma)_{kl}-R^+_{ijkm}\gamma_{ml}-R^+_{ijlm}\gamma_{km}. \label{cf}
    \end{align}
\end{lemma}

\begin{proof}
Note that 
\begin{align*}
    (\nabla^+_i\nabla^+_j\gamma)_{kl}&=e_ie_j(\gamma_{kl})-(\nabla^+_{\nabla^+_ie_j}\gamma)_{kl}-\gamma(\nabla^+_i\nabla^+_je_k,e_l)-\gamma(e_k,\nabla^+_i\nabla^+_je_l)
    \\&\kern2em -(\nabla_i^+\gamma)(\nabla_j^+e_k,e_l)-(\nabla_i^+\gamma)(e_k,\nabla_j^+e_l)-(\nabla_j^+\gamma)(\nabla_i^+e_k,e_l)-(\nabla_j^+\gamma)(e_k,\nabla_i^+e_l)
    \\&\kern2em-\gamma(\nabla_i^+e_k,\nabla_j^+e_l)-\gamma(\nabla_j^+e_k,\nabla_i^+e_l).
\end{align*}
Thus, the result follows from the definition.
\end{proof}

In this work, we mainly focus on the complex geometry. Let $(M,g,J)$ be a Hermitian manifold and $\omega\in \Lambda^{1,1}_{\mathbb{R}}$ be the associated Kähler form $\omega\in \Lambda^{1,1}_{\mathbb{R}}$ given by
\begin{align*}
    \omega(X,Y)=g(JX,Y), \quad \text{for } X,Y\in TM.
\end{align*}
We say that a Hermitian metric $g$ on $M$ is pluriclosed if 
\begin{align*}
    dd^c\omega=0,
\end{align*}
where $d^c=\sqrt{-1}(\overline{\partial}-\partial)$. The associated Lee form $\theta$ is defined by 
\begin{align}
    \theta=-d^*\omega\circ J. \label{theta}
\end{align}

In terms of the Hermitian manifold, let me define the Bismut connection formally.
\begin{defn}
Let $(M,g,J)$ be a Hermitian manifold.

\begin{itemize}
    \item The \emph{Bismut connection} $\nabla^B$ is defined by    
\begin{align}
    \langle \nabla_X^{B}Y,Z \rangle=\langle \nabla_XY,Z \rangle- \frac{1}{2}d^c\omega(X,Y,Z). \label{20}
\end{align}
In fact, $\nabla^B$ is a Hermitian connection and equals $\nabla^+$ with $H=-d^c\omega$.
\item The \emph{Bismut Ricci curvature} $\rho_{B}$ of $g$ is given by 
\begin{align}
    \rho_{B}(X,Y)\coloneqq \frac{1}{2}\langle R^{B}(X,Y)Je_i,e_i\rangle, \label{21}
\end{align}
where $R^{B}$ denotes the Riemann curvature tensor with respect to the connection $\nabla^{B}$ and $\{e_i\}$ is any orthonormal basis for $T_pM$ at any given point $p$. We say that a pluriclosed metric $g$ is a \emph{Bismut-Hermitian-Einstein metric} if $\rho_B=0$. In fact, $\rho_B^{1,1}=0$ implies $\rho_B=0$ in the compact case (c.f \cite{ye2023bismut} Theorem 1.4).
\item The \emph{Bismut scalar curvature} $S_B$ is given by 
\begin{align*}
    S_B=\tr_{\omega}\rho_B.
\end{align*}
\item The \emph{Laplace operator of the Bismut connection} $\triangle^B_f$ is defined by 
\begin{align}
    \triangle^B_f=-(\nabla^B)^{*_f}\nabla^B, 
\end{align}
where $(\nabla^B)^{*_f}$ is the formal adjoint of $\nabla^B$ with respect to $f$-twisted $L^2$ inner product (\ref{6}). In real coordinates, we get
\begin{align}
    \triangle^B_f\gamma_{ij}=\triangle_f\gamma_{ij}-H_{mjk}\nabla_m\gamma_{ik}-H_{mik}\nabla_m\gamma_{kj}-\frac{1}{4}(H_{jl}^2\gamma_{il}+H_{il}^2\gamma_{lj})+\frac{1}{2}H_{mkj}H_{mli}\gamma_{lk}. \label{LB}
\end{align}
\end{itemize}

\end{defn}

Following the computation in \cite{GRF} Chapter 8, one derives the following lemmas which will be used in later sections.
\begin{lemma}[\cite{GRF}, Lemma 8.6]
    Let $(M,g,J)$ be a Hermitian manifold and $H=-d^c\omega$. Then, for any $X\in TM$,
    \begin{align}
        \theta(X)=-\frac{1}{2}H(JX,e_i,Je_i), \label{22}
    \end{align}
    where $\{e_i\}$ is any orthonormal frame.
\end{lemma}
\begin{lemma}[\cite{GRF}, Proposition 8.10]
    Let $(M,g,J)$ be a Hermitian manifold and $H=-d^c\omega$. Then,
    \begin{align}
        \rho_B^{1,1}(\cdot,J\cdot)=\Rc-\frac{1}{4}H^2+\frac{1}{2}\mathcal{L}_{\theta_{\sharp}}g, \quad \rho_B^{2,0+0,2}(\cdot,J\cdot)=-\frac{1}{2}d^*H+\frac{1}{2}d\theta-\frac{1}{2}\iota_{\theta_{\sharp}}H. \label{23}
    \end{align}
\end{lemma}

\subsection{Generalized Einstein--Hilbert functional}
In this subsection, we discuss the generalized Einstein--Hilbert functional. Most of the contents can be found in \cite{GRF} Chapter 4 and 6.

\begin{defn}\label{D10}
Given a smooth compact manifold $M$ and a background closed 3-form $H_0$, the \emph{generalized Einstein--Hilbert functional} $\mathcal{F}: \Gamma(S^2M)\times \Omega^2 \times C^\infty (M)\to \mathbb{R}$ is given by 
\begin{align*}
    \mathcal{F}(g,b,f)=\int_M (R-\frac{1}{12}|H_0+db|^2+|\nabla f|^2)e^{-f}dV_g.
\end{align*}
Also, we define 
\begin{align*}
  \lambda(g,b)=\inf\Big\{\mathcal{F}(g,b,f)\big|\, f\in C^\infty(M),\,\int_M e^{-f}dV_g=1\Big \}.
\end{align*}

\end{defn}

\begin{remark}
 For any $(g,b)$, the minimizer $f$ is always achieved. Moreover, $\lambda$ satisfies that
\begin{align}
    \lambda(g,b)=R-\frac{1}{12}|H_0+db|^2+2\triangle f-|\nabla f|^2 \label{lam}
\end{align}
and it is the lowest eigenvalue of the Schrödinger operator $-4\triangle+R-\frac{1}{12}|H_0+db|^2.$   
\end{remark}

The goal in this work is to study the variation properties of the
functional $\lambda$ on the space of pluriclosed metrics. Let $(g_t, b_t)$ be a smooth family of pairs of Riemannian metrics and 2-forms on a smooth compact manifold $M$. Suppose
\begin{align*}
   &\frac{\partial}{\partial t}\Big|_{t=0}g_t=h,\quad \frac{\partial}{\partial t}\Big|_{t=0}b_t=K,  \quad  (g_0,b_0)=(g,b).
\end{align*}
In the following, we denote $\gamma=h-K\in\otimes^2 T^*M$ and $\gamma$ is called a general variation. The first variation formula of the generalized Einstein--Hilbert functional (c.f \cite{K} Theorem 3.3) is given by
\begin{align}
    \nonumber\frac{d}{dt}\Big|_{t=0}\lambda(g_t,b_t)&=\int_M \Big[\langle -\Rc+\frac{1}{4}H^2-\nabla^2f,h \rangle-\frac{1}{2}\langle d^*H+i_{\nabla f}H,K\rangle \Big]e^{-f}dV_g
    \\&=\int_M -\langle \gamma, \Rc^{H,f} \rangle e^{-f}dV_g, 
\end{align}
where $\Rc^{H,f}=\Rc-\frac{1}{4}H^2+\nabla^2f-\frac{1}{2}( d^*H+i_{\nabla f}H)$ is the twisted Bakry-Emery curvature. Motivated by the Ricci flow, we define the stationary points to be the steady generalized Ricci solitons. (For more details about motivations, readers can consult with \cite{GRF} and \cite{K}.)
\begin{defn}\label{soliton}
Let $M$ be a smooth manifold with a Riemannian metric $g$, a background closed 3-form $H_0$, a 2-form $b$ and a smooth function $f$ on $M$. Suppose $H=H_0+db$, then
\begin{itemize}
    \item $(M,g,b,f)$ is called a \emph{steady gradient generalized Ricci soliton} if 
\begin{align}
   0=\Rc-\frac{1}{4}H^2+\nabla^2 f, \quad 0=d_g^*H+i_{\nabla f}H.  \label{s}
\end{align}
In other words, $\Rc^B=-(\nabla^B)^2f$.
\item $(M,g,b)$ is called a \emph{generalized Einstein manifold} if
\begin{align}
  0=\Rc-\frac{1}{4}H^2, \quad 0=d_g^*H.
\end{align}

\end{itemize}
In the complex case, we say that a Hermitian manifold $(M,g,H,J,f)$ is called a \emph{pluriclosed steady soliton} if $g$ is a pluriclosed metric, $H=-d^c\omega=H_0+db$ and $(M,g,H,J,f)$ satisfies (\ref{s}).
\end{defn}

In \cite{KK} Theorem 1.1, the author further derived the second variation formula of the generalized Einstein--Hilbert which is given by
\begin{theorem}[\cite{KK} Theorem 1.1]
Let $(g,b,f)$ be a compact steady gradient generalized Ricci soliton on a smooth manifold $M$. Suppose $(g_t,b_t)$ is a one-parameter family of pairs of Riemannian metrics and 2-forms such that 
\begin{align*}
    \frac{\partial}{\partial t}\Big|_{t=0}g_t =h, \quad \frac{\partial}{\partial t}\Big|_{t=0}b_t=K,\quad  (g_0,b_0)=(g,b).
\end{align*}
Denote $\gamma=h-K$. Let $\phi$ be the unique solution of 
\begin{align*}
    \triangle_f \phi =\overline{\divg}_f\overline{\divg}_f\gamma,\quad \int_M \phi e^{-f}dV_g=0,
\end{align*}
where the definition of $\overline{\divg}_f$ is given in (\ref{bar1}) and (\ref{bar2}). The second variation of $\lambda$ at $(g,b)$ is given by 
\begin{align*}
     \nonumber\frac{d^2}{dt^2}\Big|_{t=0}\lambda=\int_M \Big\langle \gamma, \frac{1}{2}\overline{\triangle}_f\gamma+\mathring{R}^+(\gamma)+\frac{1}{2}\overline{\divg}_f^*\overline{\divg}_f\gamma+\frac{1}{2}(\nabla^+)^2\phi \Big\rangle e^{-f}dV_g,
\end{align*} 
where $\overline{\divg}_f^*$ is the formal adjoint of $\overline{\divg}_f$ with respect to (\ref{6}), $\overline{\triangle}_f$ 
 is defined in (\ref{bar3}), $\langle \mathring{R}^+(\gamma),\gamma \rangle=R^+_{iklj}\gamma_{ij}\gamma_{kl}$ and $R^+$ is the Bismut curvature given in \Cref{P3}. 

\end{theorem}

Due to the generalized slice theorem (\cite{KK}, Theorem 2.7), we decompose our variation space to be
\begin{align*}
    \otimes^2T^*M=\ker\overline{\divg}_f\oplus \text{Im}\kern0.25em\overline{\divg}_f^*. 
\end{align*}
Since $\lambda$ is diffeomorphism invariant, it suffices to consider $\gamma\in \ker\overline{\divg}_f$ to test the linear stability and we conclude that $(g,b,f)$ is linearly stable if and only if the operator $\overline{\mathcal{L}}_f\leq 0$ on $\ker\overline{\divg}_f$ where
\begin{align}
    \overline{\mathcal{L}}_f(\gamma)=\frac{1}{2}\overline{\triangle}_f\gamma+\mathring{R}^+(\gamma). 
\end{align}

At the end of this section, let's mention some examples of pluriclosed steady solitons. 

\begin{example}\label{E1}
Given complex numbers $\alpha,\beta$ satisfying $0<|\alpha|\leq |\beta|<1$, we obtain a Hopf surface 
\begin{align*}
    \mathbb{C}^2\setminus\{0\}/\langle(z_1,z_2)\rightarrow (\alpha z_1,\beta z_2) \rangle.
\end{align*}
For any $\alpha,\beta$ this is a complex manifold diffeomorphic to $S^3\times S^1$. In particular, we say that the Hopf surface is diagonal if $|\alpha|=|\beta|$. For diagonal Hopf surfaces, consider the Hopf/Boothby metric defined by the invariant Kähler form on $ \mathbb{C}^2\setminus\{0\}$ by 
\begin{align*}
    \omega_{\text{Hopf}}=\frac{\sqrt{-1}}{|z|^2}(dz_1\wedge d\overline{z}_1+dz_2\wedge d\overline{z}_2).
\end{align*}
In \cite{GRF} Proposition 8.25, one deduce that this metric is a pluriclosed metric, In fact, we observe that the metric on the universal cover  $ \mathbb{C}^2\setminus\{0\}$ is isometric to the standard cylinder $(S^3\times\mathbb{R}, g_{S^3}\oplus dt^2)$. Therefore, this metric is Bismut-flat and $\rho_B=0$. Besides, the Hopf/Boothby metric is the only compact non-Kähler example of a Bismut Hermitian-Einstein metric in complex dimension 2. Moreover, in \cite{J5} Theorem 1.1, the author also construct a non-trivial  pluriclosed steady
soliton on Hopf surface.

\end{example}

\begin{example}\label{E2}
Consider $\mathbb{C}^2\setminus\{0\}\times \mathbb{C}^2\setminus\{0\}$ endowed with the $\mathbb{C}-$action given by 
\begin{align*}
    \gamma(z,\omega)=(e^{\gamma}z,e^{\sqrt{-1}\gamma}z).
\end{align*}
This action is free and proper and the quotient space $(M,J)$ is diffeomorphic to $S^3\times S^3$. This is an example of a Calabi-Eckmann space. Moreover, $M$ is the total space of a holomorphic $T^2$ fibration over $\mathbb{CP}^1\times \mathbb{CP}^1$ and it is the product of the standard Hopf fibration $\pi_i:S^3\to \mathbb{CP}^1$. Let $\xi_i$ denotes the canonical vector fields associated to this fibration on the two factors, with $\mu_i$ the associated canonical connections satisfying $d\mu_i=\pi_i^*\omega_{FS}$. We consider the Kähler form associated to the product of round metrics, 
\begin{align*}
    \omega=\pi_1^*\omega_{FS}+\pi_2^*\omega_{FS}+\mu_1\wedge\mu_2.
\end{align*}
It follows that 
\begin{align*}
    d\omega=\pi_1^*\omega_{FS}\wedge\mu_2-\mu_1\wedge\pi_2^*\omega_{FS}.
\end{align*}
Note that $J\xi_1=\xi_2$, it shows that
\begin{align*}
    dH=-d^c\omega=\pi_1^*\omega_{FS}\wedge\mu_1-\pi_2^*\omega_{FS}\wedge\mu_2.
\end{align*}
We may view its geometric structure as a product of two copies of the $S^3$ factor of the Hopf metric so $dH=0$. It can be checked that $(M,J,\omega)$ is a Bismut-flat metric and the Lee form $\theta^{\sharp}$ is a Killing field which preserves $H$. More details can be found in \cite{GRF} Example 8.35.
\end{example}

\subsection{Pluriclosed steady solitons} 

Let $(M,g,J)$ be a Hermitian manifold. Although, we can't pick complex normal coordinates on a non-Kähler manifold, we are still able to find special coordinates to simplify our work. Here, we use the index $\alpha,\beta,\gamma,...$ to denote the index of complex coordinates. 

\begin{lemma}[\cite{J2}, Lemma 2.9]\label{L1}
    Let $(M,g,J)$ be a Hermitian manifold. Given a point $p\in M$, there exist coordinates around $p$ so that 
    \begin{align*}
        g_{\alpha\overline{\beta}}=\delta_{\alpha\beta} 
    \end{align*}
    and
    \begin{align*}
        \partial_{\alpha}g_{\beta\overline{\gamma}}+\partial_{\beta}g_{\alpha\overline{\gamma}}=0. 
    \end{align*}
\end{lemma}

Applying this lemma, we can explicitly write down the Christoffel symbols of the Bismut connection $\nabla^B$.
\begin{proposition} \label{P6}
       Let $(M,g,J)$ be a Hermitian manifold and $H=-d^c\omega \in \Lambda^{2,1}\oplus \Lambda^{1,2}$. Given a point $p\in M$, there exist coordinates around $p$ so that
    \begin{align}
       \nonumber &\nabla^B_{\alpha}\partial_{\beta}=\frac{1}{2}H_{\alpha\beta}^{\kern0.75em\tau}\partial_{\tau},\quad  \nabla_{\alpha}\partial_{\beta}=0,
\nonumber\\&\nabla^B_{\alpha}\partial_{\overline{\beta}}=H_{\alpha\overline{\beta}}^{\kern0.75em\overline{\tau}}\partial_{\overline{\tau}},\quad \nabla_{\alpha}\partial_{\overline{\beta}}=-\frac{1}{2}  H_{\alpha\overline{\beta}}^{\kern0.75em \tau}\partial_\tau+\frac{1}{2}  H_{\alpha\overline{\beta}}^{\kern0.75em \overline{\tau}}\partial_{\overline{\tau}},
\nonumber\\&\nabla^B_{\overline{\alpha}}\partial_{\beta}=H_{\overline{\alpha}\beta}^{\kern0.75em\tau}\partial_{\tau}, \quad \nabla_{\overline{\alpha}}\partial_{\beta}=\frac{1}{2}  H_{\overline{\alpha}\beta}^{\kern0.75em \tau}\partial_\tau-\frac{1}{2}  H_{\overline{\alpha}\beta}^{\kern0.75em \overline{\tau}}\partial_{\overline{\tau}},  \nonumber\\&\nabla^B_{\overline{\alpha}}\partial_{\overline{\beta}}=\frac{1}{2}H_{\overline{\alpha}\overline{\beta}}^{\kern0.75em\overline{\tau}}\partial_{\overline{\tau}},\quad \nabla_{\overline{\alpha}}\partial_{\overline{\beta}}=0, \label{26}
    \end{align}
where $\nabla^B$ is the Bismut connection defined in (\ref{20}) and $\nabla$ is the Levi-Civita connection. Moreover, 
\begin{align*}
    \partial_{\alpha}g_{\beta\overline{\gamma}}=-\frac{1}{2}H_{\alpha\beta\overline{\gamma}}, \quad \partial_{\overline{\alpha}}g_{\beta\overline{\gamma}}=\frac{1}{2}H_{\overline{\alpha}\beta\overline{\gamma}}, \quad \partial_{\alpha}g^{\beta\overline{\gamma}}=-\frac{1}{2}H_{\alpha}^{\kern0.25em\beta\overline{\gamma}}, \quad \partial_{\overline{\alpha}}g^{\beta\overline{\gamma}}=\frac{1}{2}H_{\overline{\alpha}}^{\kern0.25em\beta\overline{\gamma}}.
\end{align*}
\end{proposition}

\begin{proof}
By the definition (\ref{20}), 
\begin{align*}
    \nabla^B_{\alpha}\partial_{\overline{\beta}}-\nabla^B_{\overline{\beta}}\partial_{\alpha}=H(\partial_\alpha,\partial_{\overline{\beta}})=H_{\alpha\overline{\beta}}^{\kern0.75em\gamma}\partial_{\gamma}+H_{\alpha\overline{\beta}}^{\kern0.75em\overline{\gamma}}\partial_{\overline{\gamma}}.
\end{align*}
Since $\nabla^B J=0$, the connection $\nabla^B$ preserves types of vector fields. We see that 
\begin{align*}
    \nabla^B_{\alpha}\partial_{\overline{\beta}}=H_{\alpha\overline{\beta}}^{\kern0.75em\overline{\gamma}}\partial_{\overline{\gamma}}\quad \text{  and  }\quad \nabla^B_{\overline{\beta}}\partial_{\alpha}=H_{\overline{\beta}\alpha}^{\kern0.75em\gamma}\partial_{\gamma}.
\end{align*}
Adopt the special coordinate constructed in \Cref{L1}, we notice that $\nabla_\alpha\partial_{\beta}\in T^{0,1}$ since 
\begin{align*}
    0&=\partial_{\alpha}g_{\beta\overline{\gamma}}+\partial_{\beta}g_{\alpha\overline{\gamma}}
    \\&=g(\nabla^B_\alpha\partial_{\beta},\partial_{\overline{\gamma}})+g(\partial_{\beta},\nabla^B_\alpha\partial_{\overline{\gamma}})+g(\nabla^B_\beta\partial_{\alpha},\partial_{\overline{\gamma}})+g(\partial_{\alpha},\nabla^B_\beta\partial_{\overline{\gamma}})
    \\&=2g(\nabla_\alpha\partial_{\beta},\partial_{\overline{\gamma}}).
\end{align*}
Besides,
\begin{align*}
    \nabla^B_\alpha \partial_{\beta}=\nabla_\alpha \partial_{\beta}+\frac{1}{2}H(\partial_\alpha,\partial_\beta)=\nabla_\alpha \partial_{\beta}+\frac{1}{2}H_{\alpha\beta}^{\kern0.75em\tau}\partial_{\tau}\in T^{1,0}.
\end{align*}
We conclude that 
\begin{align*}
    \nabla_\alpha^B\partial_{\beta}=\frac{1}{2}H_{\alpha\beta}^{\kern0.75em\gamma}\partial_{\gamma} \quad \text{and} \quad \nabla_\alpha\partial_{\beta}=0.
\end{align*}
Next,
\begin{align*}
    \partial_{\alpha}g_{\beta\overline{\gamma}}&=g(\nabla^B_\alpha\partial_{\beta},\partial_{\overline{\gamma}})+g(\partial_{\beta},\nabla^B_\alpha\partial_{\overline{\gamma}})=-\frac{1}{2}H_{\alpha\beta\overline{\gamma}}.
\end{align*}
For the inverse, we compute
\begin{align*}
    0=\partial_{\alpha}(g_{\beta\overline{\gamma}}g^{\overline{\gamma}\tau})=-\frac{1}{2}H_{\alpha\beta\overline{\gamma}}g^{\overline{\gamma}\tau}+g_{\beta\overline{\gamma}}\partial_{\alpha}g^{\overline{\gamma}\tau}
\end{align*}
so 
\begin{align*}
    \partial_{\alpha}g^{\tau\overline{\gamma}}=\frac{1}{2}g^{\overline{\gamma}\beta}H_{\alpha\beta\overline{\delta}}g^{\overline{\delta}\tau}=-\frac{1}{2}H_{\alpha}^{\kern0.25em\tau\overline{\gamma}}.
\end{align*}
\end{proof}

We would like to apply this coordinate to the pluriclosed steady solitons. First, we  have the following proposition.
\begin{proposition}
Let $(M^{2n},g,H,J,f)$ be a pluriclosed steady soliton. Then,
\begin{align}
    \nabla^B(\theta-df)=\rho_B\circ J. \label{NT}
\end{align}
In particular, if $g$ is a Bismut--Hermitian--Einstein metric then $\theta-df$ is Bismut-parallel. 
\end{proposition}

\begin{proof}
As shown in \cite{J4} Proposition 4.1, suppose $(M,g,H,J,f)$ is a pluriclosed steady soliton, the vector field  $V=\frac{1}{2}(\theta^{\sharp}-\nabla f)$ satisfies 
\begin{align}
    L_VJ=0,\quad L_{JV}g=0, \quad L_{V}\omega=\rho_B^{1,1}. \label{V}
\end{align} 
In particular, $JV$ is holomorphic and Killing. For convenience, we write $V=V^\gamma\partial_\gamma+V^{\overline{\gamma}}\partial_{\overline{\gamma}}$ and compute that 
 \begin{align*}
     (\mathcal{L}_VJ)(\partial_\alpha)&=\mathcal{L}_V (J\partial_\alpha)-J(\mathcal{L}_V\partial_\alpha)
     \\&=\sqrt{-1}[V,\partial_\alpha]-J[V,\partial_\alpha]
     \\&=-\sqrt{-1}\big(\partial_\alpha(V^\gamma)\partial_\gamma+\partial_\alpha(V^{\overline{\gamma}})\partial_{\overline{\gamma}}\big)+J\big(\partial_\alpha(V^\gamma)\partial_\gamma+\partial_\alpha(V^{\overline{\gamma}})\partial_{\overline{\gamma}}\big)
     \\&=-2\sqrt{-1}\partial_\alpha(V^{\overline{\gamma}})\partial_{\overline{\gamma}}.
 \end{align*}
Similarly,
\begin{align*}
    (\mathcal{L}_VJ)(\partial_{\overline{\alpha}})=2\sqrt{-1}\partial_{\overline{\alpha}}(V^{\gamma})\partial_{\gamma}.
\end{align*}
Since $V$ is holomorphic, we get $\partial_\alpha(V^{\overline{\gamma}})=\partial_{\overline{\alpha}}(V^{\gamma})=0$. Then,
\begin{align*}
    &\partial_\alpha(V_\beta)=\partial_\alpha(g_{\beta\overline{\gamma}}V^{\overline{\gamma}})=-\frac{1}{2}H_{\alpha\beta\overline{\gamma}}V_{\gamma},
    \\&\partial_{\overline{\alpha}}(V_{\overline{\beta}})=\partial_{\overline{\alpha}}(V^\gamma g_{\gamma\overline{\beta}})=\frac{1}{2}H_{\overline{\alpha}\gamma\overline{\beta}}V_{\overline{\gamma}}.
\end{align*}
By (\ref{23}), 
\begin{align*}
    \rho_B(\partial_\alpha,J\partial_\beta)&=-\frac{1}{2}(d^*H)_{\alpha\beta}+\frac{1}{2}(d\theta)_{\alpha\beta}-\frac{1}{2}(\iota_{\theta_{\sharp}}H)_{\alpha\beta}
    \\&=\frac{1}{2}(d\theta)_{\alpha\beta}-(\iota_VH)_{\alpha\beta}.
\end{align*}
Note that 
\begin{align*}
    (d\theta)_{\alpha\beta}&=\partial_\alpha(\theta_\beta)-\partial_\beta(\theta_\alpha)
    \\&=2\partial_\alpha(V_\beta)-2\partial_\beta(V_\alpha)
    \\&=-2H_{\alpha\beta\overline{\gamma}}V_{\gamma}
    \\&=-2(\iota_VH)_{\alpha\beta}.
\end{align*}
Thus, 
\begin{align*}
    (\nabla^B_\alpha V)_\beta&=\partial_\alpha(V_\beta)-\frac{1}{2}(\iota_VH)_{\alpha\beta}
    \\&=-(\iota_VH)_{\alpha\beta}
    \\&=\frac{1}{2}\rho_B(\partial_\alpha,J\partial_\beta).
\end{align*}
On the other hand, 
\begin{align*}
    (\mathcal{L}_Vg)_{\alpha\overline{\beta}}&=V\big( g(\partial_\alpha,\partial_{\overline{\beta}}) \big)-g(\mathcal{L}_V\partial_\alpha,\partial_{\overline{\beta}})-g(\partial_\alpha,\mathcal{L}_V\partial_{\overline{\beta}})
    \\&=V^\gamma\partial_\gamma(g_{\alpha\overline{\beta}})+V^{\overline{\gamma}}\partial_{\overline{\gamma}}(g_{\alpha\overline{\beta}})-g([V,\partial_\alpha],\partial_{\overline{\beta}})-g(\partial_\alpha,[V,\partial_{\overline{\beta}}])
    \\&=-\frac{1}{2}H_{\gamma\alpha\overline{\beta}}V^\gamma+\frac{1}{2}H_{\overline{\gamma}\alpha\overline{\beta}}V^{\overline{\gamma}}+\partial_\alpha(V^\gamma)g_{\gamma\overline{\beta}}+\partial_{\overline{\beta}}(V^{\overline{\gamma}})g_{\alpha\overline{\gamma}}
    \\&=-\rho_B(J\partial_\alpha,\partial_{\overline{\beta}}),
\end{align*}
so
\begin{align*}
    \partial_\alpha(V_{\overline{\beta}})+\partial_{\overline{\beta}}(V_\alpha)&=\partial_\alpha(V^\gamma g_{\gamma\overline{\beta}})+\partial_{\overline{\beta}}(V^{\overline{\gamma}}g_{\alpha\overline{\gamma}})
    \\&=-\frac{1}{2}H_{\alpha\gamma\overline{\beta}}V^\gamma+\frac{1}{2}H_{\overline{\beta}\alpha\overline{\gamma}}V^{\overline{\gamma}}+\partial_\alpha(V^\gamma)g_{\gamma\overline{\beta}}+\partial_{\overline{\beta}}(V^{\overline{\gamma}})g_{\alpha\overline{\gamma}}
    \\&=H_{\alpha\overline{\beta}\gamma}V^\gamma-H_{\alpha\overline{\beta}\overline{\gamma}}V^{\overline{\gamma}}-\rho_B(J\partial_\alpha,\partial_{\overline{\beta}}).
\end{align*}
Using the fact that $JV$ is also a Killing field, we get

\begin{align*}
    (\mathcal{L}_{JV}g)_{\alpha\overline{\beta}}&=JV\big( g(\partial_\alpha,\partial_{\overline{\beta}}) \big)-g([JV,\partial_\alpha],\partial_{\overline{\beta}})-g(\partial_\alpha,[JV,\partial_{\overline{\beta}}])
    \\&=\sqrt{-1}V^\gamma\partial_\gamma(g_{\alpha\overline{\beta}})-\sqrt{-1}V^{\overline{\gamma}}\partial_{\overline{\gamma}}(g_{\alpha\overline{\beta}})+\sqrt{-1}\partial_\alpha(V^\gamma)g_{\gamma\overline{\beta}}-\sqrt{-1}\partial_{\overline{\beta}}(V^{\overline{\gamma}})g_{\alpha\overline{\gamma}}
    \\&=\sqrt{-1}\Big(-\frac{1}{2}H_{\gamma\alpha\overline{\beta}}V^\gamma-\frac{1}{2}H_{\overline{\gamma}\alpha\overline{\beta}}V^{\overline{\gamma}}+\partial_\alpha(V^\gamma)g_{\gamma\overline{\beta}}-\partial_{\overline{\beta}}(V^{\overline{\gamma}})g_{\alpha\overline{\gamma}}\Big)
    \\&=0,
\end{align*}
so
\begin{align*}
   \partial_\alpha(V_{\overline{\beta}})-\partial_{\overline{\beta}}(V_\alpha)&=\partial_\alpha(V^\gamma g_{\gamma\overline{\beta}})-\partial_{\overline{\beta}}(V^{\overline{\gamma}}g_{\alpha\overline{\gamma}})
   \\&=-\frac{1}{2}H_{\alpha\gamma\overline{\beta}}V^\gamma-\frac{1}{2}H_{\overline{\beta}\alpha\overline{\gamma}}V^{\overline{\gamma}}+\partial_\alpha(V^\gamma)g_{\gamma\overline{\beta}}-\partial_{\overline{\beta}}(V^{\overline{\gamma}})g_{\alpha\overline{\gamma}}
\\&=H_{\gamma\alpha\overline{\beta}}V^\gamma+H_{\overline{\gamma}\alpha\overline{\beta}}V^{\overline{\gamma}}.
\end{align*}
It implies that 
\begin{align*}
    \partial_\alpha(V_{\overline{\beta}})=H_{\alpha\overline{\beta}\gamma}V^\gamma-\frac{1}{2}\rho_B(J\partial_\alpha,\partial_{\overline{\beta}}), \quad \partial_{\overline{\beta}}(V_\alpha)=-H_{\alpha\overline{\beta}\overline{\gamma}}V^{\overline{\gamma}}-\frac{1}{2}\rho_B(J\partial_\alpha,\partial_{\overline{\beta}}),
\end{align*}
i.e.,
\begin{align*}
    (\nabla_\alpha^BV)_{\overline{\beta}}=-\frac{1}{2}\rho_B(J\partial_\alpha,\partial_{\overline{\beta}}), \quad (\nabla^B_{\overline{\beta}}V)_\alpha=-\frac{1}{2}\rho_B(J\partial_{\overline{\beta}},\partial_\alpha).
\end{align*}
\end{proof}

Secondly, we express the Lee form $\theta$ and $\rho_B$ in terms of the complex coordinate and we will use these results later.
\begin{lemma}
Let $(M,g,J)$ be a Hermitian manifold and $H=-d^c\omega \in \Lambda^{2,1}\oplus \Lambda^{1,2}$. Then,    
\begin{align}
    H_{\overline{\alpha}\alpha\beta}=\theta_\beta, \quad H_{\overline{\alpha}\alpha\overline{\beta}}=-\theta_{\overline{\beta}}, \label{28}
\end{align}
and
\begin{align}
    \rho_B(\partial_\alpha,J\partial_{\overline{\beta}})=R^B_{\alpha\overline{\beta}\tau\overline{\tau}},\quad  \rho_B(\partial_\alpha,J\partial_\beta)=-R^B_{\alpha\beta\tau\overline{\tau}}. \label{29}
\end{align}
\end{lemma}
\begin{proof}
According to (\ref{21}) and (\ref{22}), 
\begin{align*}
        \theta(X)=-\frac{1}{2}H(JX,e_i,Je_i), \quad  \rho_B(X,Y)= \frac{1}{2}\langle R^B(X,Y)Je_i,e_i\rangle
    \end{align*}
for any orthonormal frame $\{e_1,e_2,...,e_{2n}\}$.  In our case, we construct an orthonormal frame $\{u_1,u_2,...,u_n,Ju_1,Ju_2,...,Ju_n\}$ by
\begin{align*}
    u_{\alpha}=\frac{\partial_\alpha+\partial_{\overline{\alpha}}}{\sqrt{2}}, \quad Ju_{\alpha}=\frac{\partial_{\overline{\alpha}}-\partial_\alpha}{\sqrt{-2}}.
\end{align*}
Therefore,
\begin{align*}
    \theta(X)&=-\sum_{\alpha=1}^n H(JX,u_\alpha,Ju_\alpha),
    \\ \rho_B(X,Y)&=\sum_{\alpha=1}^n\langle R^B(X,Y)Ju_\alpha, u_\alpha\rangle.
\end{align*}
Then,
\begin{align*}
\sum_{\alpha=1}^n H_{\overline{\alpha}\alpha\beta}= \sum_{\alpha=1}^n H(\frac{u_\alpha+\sqrt{-1}Ju_\alpha}{\sqrt{2}},\frac{u_\alpha-\sqrt{-1}Ju_\alpha}{\sqrt{2}},\partial_\beta)=\sum_{\alpha=1}^n \sqrt{-1}H(Ju_\alpha,u_\alpha,\partial_\beta)=-\sqrt{-1}\theta(J\partial_\beta)=\theta_\beta,
\end{align*}
and
\begin{align*}
     \rho_B(\partial_\alpha,J\partial_{\overline{\beta}})=\sum_{\tau=1}^n \frac{1}{2\sqrt{-1}} \langle R^B(\partial_\alpha,J\partial_{\overline{\beta}}) (\partial_{\overline{\tau}}-\partial_\tau),\partial_\tau+\partial_{\overline{\tau}}\rangle=\sum_{\tau=1}^n R^B_{\alpha\overline{\beta}\tau\overline{\tau}}.
\end{align*}
\end{proof}

Lastly, we show that $f$ is invariant when $g$ is a Bismut--Hermitian--Einstein metric.
\begin{lemma}\label{LV}
    Suppose $(M,g,H)$ is a compact Bismut--Hermitian--Einstein manifold, then 
    \begin{align*}
        \mathcal{L}_{\theta^\sharp-\nabla f}f=0.
    \end{align*}
\end{lemma}
\begin{proof}
In Bismut--Hermitian--Einstein case, we note that $\theta^\sharp-\nabla f$ is a holomorphic Killing vector field by (\ref{V}). Thus, taking the Lee derivative of $\theta^\sharp-\nabla f$ to the equation $R-\frac{1}{4}|H|^2+\triangle f=0$ deduces that $\triangle(\mathcal{L}_{\theta^\sharp-\nabla f}f)=0$. Then, we conclude that $\mathcal{L}_{\theta^\sharp-\nabla f}f$ is a constant. By maximum principle, this constant must be 0 in the compact case.    
\end{proof}

\section{The stability of Bismut-Hermitian-Einstein Manifolds}

\subsection{Essential infinitesimal complex structure}
\begin{defn}
 Let $(M,g,J)$ be a complex manifold. A tensor $I\in \text{End}(TM)$ is called an \emph{infinitesimal variation of complex structure} if 
 \begin{align*}
     IJ+JI=0, \quad \text{and}\quad N'_J(I)=0,
 \end{align*}
 where $N$ denotes the Nijenhuis tensor. Moreover, we say that an infinitesimal variation of complex structure is \emph{$f$-essential} if 
 \begin{align*}
     \int_M \langle \mathcal{L}_ZJ,I \rangle e^{-f}dV_g=0\quad \text{  for any vector $Z$.}
 \end{align*}
\end{defn}

In the following, we use complex coordinates $\{\partial_\alpha,\partial_{\overline{\beta}}\}$ and write
\begin{align*}
    I=I_{\alpha}^\beta dz^\alpha\otimes \partial_{\beta}+I_{\overline{{\alpha}}}^\beta dz^{\overline{\alpha}}\otimes \partial_{\beta}+I_{\alpha}^{\overline{\beta}} dz^\alpha\otimes \partial_{\overline{\beta}}+I_{\overline{\alpha}}^{\overline{\beta}} dz^{\overline{\alpha}}\otimes \partial_{\overline{\beta}}.
\end{align*}
From Koiso \cite{Ko}, we see that if $I$ is an infinitesimal variation of complex structure then 
\begin{align*}
    I_{\alpha}^\beta=0 \quad \text{and} \quad \partial_\alpha I_\beta^{\overline{\gamma}}=\partial_\beta I_{\alpha}^{\overline{\gamma}}.
\end{align*}

\begin{lemma}\label{leta}
Let $(M,g,J)$ be a Hermitian manifold and $I$ be an $f$-essential infinitesimal variation of complex structure. Define    
a 2-tensor $\eta$ by 
\begin{align*}
    \eta(X,Y)=\omega(X,IY)=g(JX,IY), \text{ for any vectors $X$ and $Y$.}
\end{align*}
If $\eta$ is symmetric then $\eta$ is anti-Hermitian, $\divg^B_f\eta=0$. Moreover, in the coordinate we constructed in \Cref{L1}, we get
\begin{align}
    (\nabla^B_\alpha\eta)_{\beta\gamma}-(\nabla^B_\beta\eta)_{\alpha\gamma}=-H_{\alpha\beta\overline{\tau}}\eta_{\tau\gamma}+H_{\beta\gamma\overline{\tau}}\eta_{\tau\alpha}+H_{\gamma\alpha\overline{\tau}}\eta_{\tau\beta}. \label{k2} 
\end{align}
\end{lemma}
\begin{proof}
Fix a given point $p$ and choose coordinates constructed in \Cref{L1}. First, we note that 
\begin{align*}
    &\eta_{\alpha\beta}=\omega(\partial_\alpha,I\partial_{\beta})=\omega(\partial_\alpha,I_{\beta}^{\overline{\gamma}}\partial_{\overline{\gamma}})=\sqrt{-1}I_{\alpha\beta},
\\&\eta_{\alpha\overline{\beta}}=\omega(\partial_\alpha,I\partial_{\overline{\beta}})=\omega(\partial_\alpha,I_{\overline{\beta}}^\gamma\partial_\gamma)=0
\end{align*}
since $\omega\in \Lambda^{1,1}$. In other words, $\eta$ is anti-Hermitian. Next, we observe that 
\begin{align*}
    0&=\int_M \langle \mathcal{L}_ZJ,I \rangle e^{-f}dV_g
    \\&=\int_M \langle \divg_f\eta,Z \rangle+\langle \mathcal{L}_Z\omega, \eta\circ J \rangle e^{-f}dV_g.
\end{align*}
Since $\eta\circ J$ is symmetric, $\langle \mathcal{L}_Z\omega, \eta\circ J \rangle=0$. Then, $\divg_f\eta=\divg_f^B\eta=0$. Lastly, we compute
 \begin{align*}
     (\nabla^B_\alpha\eta)_{\beta\gamma}&=\partial_\alpha(\eta_{\beta\gamma})-\eta(\nabla^B_\alpha\partial_\beta,\partial_\gamma)-\eta(\partial_\beta,\nabla^B_\alpha\partial_\gamma)
     \\&=\partial_\alpha\big( g(J\partial_\gamma,I_\beta^{\overline{\tau}}\partial_{\overline{\tau}})  \big)-\frac{1}{2}H_{\alpha\beta\overline{\tau}}\eta_{\tau\gamma}-\frac{1}{2}H_{\alpha\gamma\overline{\tau}}\eta_{\beta\tau}
     \\&=g(J\partial_\gamma,(\partial_\alpha I_\beta^{\overline{\tau}})\partial_{\overline{\tau}})+g(J\partial_\gamma,I_\beta^{\overline{\tau}}\nabla^B_\alpha\partial_{\overline{\tau}})-\frac{1}{2}H_{\alpha\beta\overline{\tau}}\eta_{\tau\gamma}
     \\&=g(J\partial_\gamma,(\partial_\alpha I_\beta^{\overline{\tau}})\partial_{\overline{\tau}})+H_{\alpha\overline{\tau}\gamma}\eta_{\beta\tau}-\frac{1}{2}H_{\alpha\beta\overline{\tau}}\eta_{\tau\gamma}
 \end{align*}
 Thus,
 \begin{align*}
     (\nabla^B_\alpha\eta)_{\beta\gamma}-(\nabla^B_\beta\eta)_{\alpha\gamma}=H_{\alpha\overline{\tau}\gamma}\eta_{\beta\tau}-H_{\alpha\beta\overline{\tau}}\eta_{\tau\gamma}-H_{\beta\overline{\tau}\gamma}\eta_{\alpha\tau}.
 \end{align*}
\end{proof}

\subsection{Decomposition of Variation Space}

In this section, our goal is to study the second variation of the generalized Einstein--Hilbert functional on a pluriclosed steady soliton.  Let $(M,g,H,J,f)$ be a pluriclosed steady soliton with a closed 3-form $H=-d^c\omega$. Consider a family $(M,g_t,H_t,J_t,f_t)$ with 
\begin{align*}
    H_t=H_0+db_t,\quad (M,g_0,H_0,J_0,b_0)=(M,g,H,J,0)
\end{align*}
and $f_t$ is the minimizer of $\lambda(g_t,b_t)$. Denote 
\begin{align*}
   \frac{\partial }{\partial t}\big|_{t=0} g=h, \quad  \frac{\partial }{\partial t}\big|_{t=0} b=\beta,\quad   \frac{\partial }{\partial t}\big|_{t=0}\omega =\phi, \quad  \frac{\partial }{\partial t}\big|_{t=0} J =I,
\end{align*}
where $I$ is an $f$-essential infinitesimal variation of complex structure. We get the following proposition.

\begin{proposition}
For any general variation $\gamma\in\ker\overline{\divg}_f$, there exists a unique pair $(\xi,\eta)$ such that
\begin{align}
    \gamma(X,Y)=\xi(X,JY)+\eta(X,Y),  \label{k3}
\end{align}
where $\xi$ is a 2-form with $(d^B_f)^*\xi=0$ and $\eta$ is a symmetric 2-tensor with $\divg_f^B\eta=0$ and satisfies (\ref{k2}).
\end{proposition}

\begin{proof}
Following the notation written above,  we note that $g_t$ are compatible with $J_t$, so 
\begin{align*}
     \frac{\partial }{\partial t}\big|_{t=0} g(X,Y)&= \frac{\partial }{\partial t}\big|_{t=0}\big(\omega_t(X,J_tY) \big)
     \\&=\phi(X,JY)+\eta(X,Y), \quad \text{for all vector fields $X,Y$.}
\end{align*}
From the definition, a general variation $\gamma=h-\beta$ is given by
\begin{align*}
    \gamma(X,Y)= \phi(X,JY)+\eta(X,Y)-\beta(X,Y).
\end{align*}
We consider
\begin{align*}
    &\Tilde{\phi}^{1,1}(X,Y)=\phi^{1,1}(X,Y)+\beta^{1,1}(X,JY), \quad \Tilde{\phi}^{2,0}= \Tilde{\phi}^{0,2}=0,\quad 
    \\&\Tilde{\beta}^{2,0}(X,JY)=\phi^{2,0}(X,Y)+\beta^{2,0}(X,JY), \quad \Tilde{\beta}^{0,2}=\overline{\Tilde{\beta}^{2,0}},\quad \Tilde{\beta}^{1,1}=0.
\end{align*}
Since 
\begin{align*}
    \phi(X,JY)-\beta(X,Y)=\Tilde{\phi}(X,JY)-\Tilde{\beta}(X,Y),
\end{align*}    
in the following, we may assume $\phi\in \Lambda^{1,1}$ and $\beta\in \Lambda^{2,0+0,2}$. Therefore, $\eta$ is symmetric and $\divg_f^B\eta=0$ by \Cref{leta}. Note that $\gamma\in\ker\overline{\divg}_f$,
\begin{align*}
    \divg_f(\phi\circ J)_l=\frac{1}{2}\beta_{ij}H_{ijl},\quad d^*_f\beta=0,
\end{align*}
by \Cref{RRRRR}. Define $\xi=\phi+\beta\circ J$, we deduce that $(d^B_f)^*\xi=0$. Besides, $\eta$ is symmetric and anti-Hermitian. Therefore, $\xi$ and $\eta$ are orthogonal in the sense of twisted $L^2$ inner product (\ref{6}). In other words, the pair $(\xi,\eta)$ is unique.

\end{proof}

In the following, we assume that the general variation $\gamma\in \ker\overline{\divg}_f$. The conclusion in section 2.3 deduces that the second variation formula is given by 
\begin{align*}
     \nonumber\frac{d^2}{dt^2}\Big|_{t=0}\lambda=\int_M\langle \gamma, \overline{\mathcal{L}}_f(\gamma) \rangle e^{-f}dV_g.
\end{align*} 
By (\ref{bar4}) and (\ref{LB}), we have
\begin{align}
    \nonumber \overline{\mathcal{L}}_f(\gamma)_{ij}&=\frac{1}{2}\overline{\triangle}_f\gamma_{ij}+\mathring{R}^B(\gamma)_{ij}
    \\&=\frac{1}{2}\triangle_f^B\gamma_{ij}+\mathring{R}^B(\gamma)_{ij}-\nabla^B_m\gamma_{kj}H_{mki}-\frac{1}{2}H^2_{ki}\gamma_{kj}.
\end{align}
use the decomposition (\ref{k3}), it suffices to compute the second variation in three parts
\begin{align*}
     \int_M \langle \overline{\mathcal{L}}_f(\xi\circ J),\eta  \rangle e^{-f}dV_g, \quad  \int_M \langle \overline{\mathcal{L}}_f(\xi\circ J),\xi\circ J  \rangle e^{-f}dV_g,\quad  \int_M \langle \overline{\mathcal{L}}_f\eta,\eta  \rangle e^{-f}dV_g.   
\end{align*}
We first compute the cross term.

\begin{lemma}\label{L6}
Suppose $\xi$ is a 2-form and $\eta$ is anti-Hermitian symmetric 2-tensor with $\divg^B_f\eta=0$ on a complex manifold $(M,g,J)$. Then,
\begin{align*}
    \int_M \langle \overline{\mathcal{L}}_f(\xi\circ J),\eta  \rangle e^{-f}dV_g =\int_M \langle \overline{\mathcal{L}}_f(\eta),\xi\circ J \rangle e^{-f}dV_g=\int_M \frac{1}{6} \langle d\xi, C \rangle e^{-f}dV_g,
\end{align*}
where $C$ is a three form defined by 
\begin{align}
    C_{ijk}=H_{ijl}(\eta\circ J)_{lk}+H_{jkl}(\eta\circ J)_{li}+H_{kil}(\eta\circ J)_{lj}. \label{C}
\end{align}
\end{lemma}

\begin{proof}
We compute 
\begin{align*}
    \int_M \langle \overline{\mathcal{L}}_f(\xi\circ J),\eta  \rangle e^{-f}dV_g &=\int_M \Big( \frac{1}{2}\triangle_f^B (\xi\circ J)_{ij}\eta_{ij}+\mathring{R}^B(\xi\circ J)_{ij}\eta_{ij}-\nabla^B_m(\xi\circ J)_{kj}H_{mki}\eta_{ij}-\frac{1}{2}H^2_{ki}(\xi\circ J)_{kj}\eta_{ij}\Big)  e^{-f}dV_g
    \\&=-\int_M \Big( \frac{1}{2}\triangle_f^B\xi_{ij}(\eta\circ J)_{ij}+\mathring{R}^B(\xi)_{ij}(\eta\circ J)_{ij}-\nabla^B_m\xi_{kj}H_{mki}(\eta\circ J)_{ij}-\frac{1}{2}H^2_{ki}\xi_{kj}(\eta\circ J)_{ij} \Big) e^{-f}dV_g.
\end{align*}
Using the Bianchi identity (\ref{10}), 
\begin{align*}
    R^B_{iklj}+R^B_{klij}+R^B_{likj}=(\nabla_j^BH)_{ikl}-H_{ikm}H_{ljm}-H_{klm}H_{ijm}-H_{lim}H_{kjm}.
\end{align*}
We compute
\begin{align*}
    \langle \mathring{R}^B(\xi),\eta\circ J  \rangle&= R^B_{iklj}\xi_{kl}(\eta\circ J)_{ij}
    \\&=\big(-R^B_{klij}-R^B_{likj}+(\nabla_j^BH)_{ikl}-H_{ikm}H_{ljm}-H_{klm}H_{ijm}-H_{lim}H_{kjm}   \big)\xi_{kl}(\eta\circ J)_{ij}
    \\&=-R^B_{iklj}\xi_{kl}(\eta\circ J)_{ij}+(\nabla_j^BH)_{ikl}\xi_{kl}(\eta\circ J)_{ij}
    \\&=-\langle \mathring{R}^B(\xi),\eta\circ J  \rangle+(\nabla_j H)_{ikl}\xi_{kl}(\eta\circ J)_{ij},
\end{align*}
here we notice that $\eta\circ J$ is symmetric since $\eta$ is anti-Hermitian and symmetric. Next, we compute
\begin{align*}
    -\nabla^B_m\xi_{kj}H_{mki}(\eta\circ J)_{ij}=-\nabla_m\xi_{kj}H_{mki}(\eta\circ J)_{ij}+\frac{1}{2}H_{il}^2\xi_{lj}(\eta\circ J)_{ij}.
\end{align*}
We conclude that 
\begin{align*}
    \int_M \langle \overline{\mathcal{L}}_f(\xi\circ J),\eta \rangle e^{-f}dV_g &=-\int_M\Big( \langle \mathring{R}^B(\xi),\eta\circ J  \rangle-\nabla^B_m\xi_{kj}H_{mki}(\eta\circ J)_{ij}-\frac{1}{2}H^2_{ki}\xi_{kj}(\eta\circ J)_{ij}\Big)  e^{-f}dV_g
    \\&=-\int_M \Big( \frac{1}{2}(\nabla_j H)_{ikl}\xi_{kl}(\eta\circ J)_{ij}-\nabla_m\xi_{kj}H_{mki}(\eta\circ J)_{ij}\Big)  e^{-f}dV_g
    \\&= \int_M (\nabla_k\xi_{lj}+\frac{1}{2}\nabla_j\xi_{kl})H_{ikl}(\eta\circ J)_{ij}  e^{-f}dV_g
    \\&=\int_M \Big( \frac{1}{2}(d\xi)_{klj}H_{ikl}(\eta\circ J)_{ij} \Big) e^{-f}dV_g
    \\&=\int_M \frac{1}{6} \langle d\xi, C \rangle e^{-f}dV_g,
\end{align*}
since $\divg^B_f(\eta\circ J)=\divg_f(\eta\circ J)=0$.

\end{proof}

\subsection{Anti-symmetric Part}

In this section, we focus on the anti-symmetric part $\xi$. First, we observe that 
\begin{align*}
    \int_M\langle& \xi\circ J, \overline{\mathcal{L}}_f(\xi\circ J) \rangle e^{-f}dV_g
    \\&=\int_M \Big( \frac{1}{2}\triangle_f^B (\xi\circ J)_{ij}(\xi\circ J)_{ij}+\mathring{R}^B(\xi\circ J)_{ij}(\xi\circ J)_{ij}-\nabla^B_m(\xi\circ J)_{kj}H_{mki}(\xi\circ J)_{ij}-\frac{1}{2}H^2_{ki}(\xi\circ J)_{kj}(\xi\circ J)_{ij}\Big)  e^{-f}dV_g
    \\&=\int_M \Big( \frac{1}{2}\triangle_f^B\xi_{ij}\xi_{ij}+\mathring{R}^B(\xi)_{ij}\xi_{ij}-\nabla^B_m\xi_{kj}H_{mki}\xi_{ij}-\frac{1}{2}H^2_{ki}\xi_{kj}\xi_{ij}\Big)  e^{-f}dV_g.
\end{align*}
Next, we compute the Hodge Laplacian of Bismut connection $\triangle^B_{d,f}=-d^B (d^B_f)^*-(d^B_f)^*d^B$.

\begin{lemma}\label{LD}
 Suppose $(M,g,H,f)$ is a steady gradient generalized Ricci soliton. The Hodge Laplacian of Bismut connection $\triangle^B_{d,f}=-d^B (d^B_f)^*-(d^B_f)^*d^B$ on 2-form $\xi$ is given by
 \begin{align}
     \triangle^B_{d,f}\xi_{ij}=\triangle^B_f\xi_{ij}+\mathring{R}^B(\xi)_{ij}-\mathring{R}^B(\xi)_{ji}-H_{lik}\nabla^B_k\xi_{jl}-H_{ljk}\nabla^B_k\xi_{li}.
 \end{align}
\end{lemma}
\begin{proof}
We compute
\begin{align*}
    -(d^B_f)^* d^B\xi_{ij}&=\big(\nabla^B_l(d^B\xi)\big)(e_l,e_i,e_j)-\nabla_lf (d^B\xi)_{lij}
    \\&=\partial_l\big( (d^B\xi)_{lij} \big)-(d^B\xi)(e_l,\nabla^B_le_i,e_j)-(d^B\xi)(e_l,e_i\nabla^B_le_j)-\nabla_lf (d^B\xi)_{lij}
    \\&=\partial_l\big((\nabla^B_l\xi)_{ij}+(\nabla^B_i\xi)_{jl}+(\nabla^B_j\xi)_{li}\big)-\frac{1}{2}H_{lik}(d^B\xi)_{lkj}-\frac{1}{2}H_{ljk}(d^B\xi)_{lik}-\nabla_lf\big((\nabla^B_l\xi)_{ij}+(\nabla^B_i\xi)_{jl}+(\nabla^B_j\xi)_{li}\big)
    \\&=\triangle^B_f\xi_{ij}+(\nabla^B_l\nabla^B_i\xi)_{jl}+(\nabla^B_l\nabla^B_j\xi)_{li}-\nabla_lf\big((\nabla^B_i\xi)_{jl}+(\nabla^B_j\xi)_{li}\big).
\end{align*}
Using the commutator formula (\ref{cf}), 
\begin{align*}
    (\nabla^B_l\nabla^B_i\xi)_{jl}&=(\nabla^B_i\nabla^B_l\xi)_{jl}-H_{lik}(\nabla^B_k\xi)_{jl}-R^B_{lijk}\xi_{kl}-R^B_{lilk}\xi_{jk}
    \\&=\partial_i\big( (\nabla^B_l\xi)_{jl} \big)+H_{ilk}(\nabla^B_k\xi)_{jl}-\frac{1}{2}H_{ijk}(\nabla^B_l\xi)_{kl}+\mathring{R}^B(\xi)_{ij}+R^B_{ik}\xi_{jk}
    \\&=\partial_i\big( ((d^B_f)^*\xi)_j-\nabla_lf\xi_{lj}  \big)+H_{ilk}(\nabla^B_k\xi)_{jl}-\frac{1}{2}H_{ijk}\big( ((d^B_f)^*\xi)_k-\nabla_lf\xi_{lk}  \big)+\mathring{R}^B(\xi)_{ij}+R^B_{ik}\xi_{jk}
    \\&=\nabla_i^B((d^B_f)^*\xi)_j-\nabla_i^B\nabla_l^Bf \xi_{lj}-\frac{1}{2}H_{ilk}\nabla_kf \xi_{lj}-\nabla_l f \partial_i(\xi_{lj})+\frac{1}{2}H_{ijk}\nabla_lf \xi_{lk}+H_{ilk}(\nabla^B_k\xi)_{jl}+\mathring{R}^B(\xi)_{ij}+R^B_{ik}\xi_{jk}
    \\&=\nabla_i^B(d^B_f)^*\xi)_j+\mathring{R}^B(\xi)_{ij}+H_{ilk}(\nabla^B_k\xi)_{jl}-\nabla_lf (\nabla^B_i\xi)_{lj},
\end{align*}
where we used the assumption (\ref{s}) that $\Rc^B=-(\nabla^B)^2 f$ Therefore,
\begin{align*}
    -(d^B_f)^*d^B\xi_{ij}&=\triangle^B_f\xi_{ij}+d^B (d^B_f)^*\xi_{ij}+\mathring{R}^B(\xi)_{ij}-\mathring{R}^B(\xi)_{ji}+H_{ilk}(\nabla^B_k\xi)_{jl}-H_{jlk}(\nabla^B_k\xi)_{il},
\end{align*}
which finish the proof.

\end{proof}

Now, we can state the following proposition.
\begin{proposition}
Let $(M,g,H,J,f)$ be a compact pluriclosed steady soliton and $\xi$ is a 2-form with $(d^B_f)^*\xi=0$. Then,      
\begin{align}
    \int_M\langle \xi\circ J, \overline{\mathcal{L}}_f(\xi\circ J) \rangle e^{-f}dV_g=-2\|d_f^*\xi\|_f^2-\frac{1}{6}\|d\xi\|_f^2,
\end{align}
where  $\|\cdot\|_f$ denotes the norm of $f$-twisted $L^2$ inner product (\ref{6}).
\end{proposition}

\begin{proof}
Using the \Cref{LD}, 
\begin{align*}
    \frac{1}{2}\langle \triangle^B_{d,f}\xi,\xi  \rangle=\frac{1}{2}\langle \triangle^B_{f}\xi,\xi  \rangle+\langle \mathring{R}^B(\xi),\xi \rangle+H_{ilk}(\nabla^B_k\xi)_{jl}\xi_{ij}.
\end{align*}
Therefore, 
\begin{align*}
     \int_M\langle\xi\circ J, \overline{\mathcal{L}}_f(\xi\circ J) \rangle e^{-f}dV_g&=\int_M \Big(  \frac{1}{2}\langle \triangle^B_{d,f}\xi,\xi  \rangle-2\nabla^B_m\xi_{kj}H_{mki}\xi_{ij}-\frac{1}{2}H^2_{ki}\xi_{kj}\xi_{ij}\Big)  e^{-f}dV_g
     \\&=\int_M \Big( -\frac{1}{6}|d^B\xi|^2-2\nabla^B_m\xi_{kj}H_{mki}\xi_{ij}-\frac{1}{2}H^2_{ki}\xi_{kj}\xi_{ij}\Big)  e^{-f}dV_g.
\end{align*}
where we used the convention that 
\begin{align*}
    \int_M \langle \triangle^B_{d,f}\xi,\xi  \rangle e^{-f}dV_g&= \int_M -\langle(d^B)^* d^B\xi,\xi \rangle e^{-f}dV_g
    \\&=\int_M (\nabla^B_{l}d^B\xi)_{lij}\xi_{ij}e^{-f}dV_g
    \\&=\int_M -(d^B\xi)_{lij}\nabla^B_{l}\xi_{ij}e^{-f}dV_g
    \\&=-\frac{1}{3}\|d^B\xi\|^2_f
\end{align*}

Note that 
\begin{align*}
    (d^B\xi)_{mij}=(d\xi)_{mij}-H_{mik}\xi_{kj}-H_{ijk}\xi_{km}-H_{jmk}\xi_{ki},
\end{align*}
so 
\begin{align*}
    |d^B\xi|^2=|d\xi|^2-6(d\xi)_{mij}H_{mik}\xi_{kj}+3H^2_{kl}\xi_{lj}\xi_{kj}-6H_{iab}H_{jcb}\xi_{ij}\xi_{ac}.
\end{align*}
Then, 
\begin{align*}
     \int_M\langle\xi\circ J, \overline{\mathcal{L}}_f(\xi\circ J) \rangle e^{-f}dV_g&=\int_M \Big((-\frac{1}{6}|d^B\xi|^2-2\nabla^B_m\xi_{kj}H_{mki}\xi_{ij}-\frac{1}{2}H^2_{ki}\xi_{kj}\xi_{ij} \Big) e^{-f}dV_g
     \\&=\int_M\Big(-\frac{1}{6}|d\xi|^2+(d\xi)_{mij}H_{mik}\xi_{kj}-2\nabla^B_m\xi_{kj}H_{mki}\xi_{ij}-H^2_{ki}\xi_{kj}\xi_{ij}+H_{iab}H_{jcb}\xi_{ij}\xi_{ac}\Big)  e^{-f}dV_g
     \\&=\int_M \Big(-\frac{1}{6}|d\xi|^2+(d\xi)_{mij}H_{mik}\xi_{kj} -2\nabla_m\xi_{kj}H_{mki}\xi_{ij} \Big) e^{-f}dV_g
     \\&=\int_M \Big(-\frac{1}{6}|d\xi|^2+ \nabla_j\xi_{mk}H_{mki}\xi_{ij}\Big)  e^{-f}dV_g
     \\&=\int_M \Big(-\frac{1}{6}|d\xi|^2-\nabla_jH_{mki}\xi_{mk}\xi_{ij}-(d_f^*\xi)_i H_{mki}\xi_{mk}\Big) e^{-f}dV_g
     \\&=-\frac{1}{6}\|d\xi\|_f^2-2\|d_f^*\xi\|_f^2.
\end{align*}
Here we used the assumption that $0=(d^B_f)^*\xi_l=d^*_f\xi_l+\frac{1}{2}H_{mlk}\xi_{mk}$ and $dH=0$.

\end{proof}

\subsection{Symmetric Part}
In this subsection, we use the complex coordinate given in \Cref{P6} to analyze symmetric part. First, we derive the following lemmas.

\begin{lemma}
Let $(M^{2n},g,H,J,f)$ be a pluriclosed steady soliton. Then,
\begin{align}
    \nonumber&(\nabla^B_{\overline{\alpha}} H)_{\beta\overline{\tau}\alpha} =\frac{1}{2}(\mathcal{L}_{\theta^\sharp}g)_{\beta\overline{\tau}}+\frac{1}{2}(\iota_{\nabla f}H)_{\beta\overline{\tau}}+\big(\theta_{\overline{\delta}}H_{\beta\overline{\tau}\delta}-\theta_{\delta}H_{\beta\overline{\tau}\overline{\delta}}+H_{\overline{\tau}\alpha\delta}H_{\beta\overline{\alpha}\overline{\delta}}\big)
    \nonumber \\&(\nabla_{\overline{\alpha}} H)_{\beta\overline{\tau}\alpha}=\frac{1}{2}(\mathcal{L}_{\theta^\sharp}g)_{\beta\overline{\tau}}+\frac{1}{2}(\iota_{\nabla f}H)_{\beta\overline{\tau}}+\frac{1}{2}\big(\theta_{\overline{\delta}}H_{\beta\overline{\tau}\delta}-\theta_{\delta}H_{\beta\overline{\tau}\overline{\delta}}+H_{\overline{\tau}\alpha\delta}H_{\beta\overline{\alpha}\overline{\delta}}\big)
    \nonumber\\&(\nabla^B_\alpha H)_{\beta\overline{\tau}\overline{\alpha}}=-\frac{1}{2}(\mathcal{L}_{\theta^\sharp}g)_{\beta\overline{\tau}}+\frac{1}{2}(\iota_{\nabla f}H)_{\beta\overline{\tau}}-\big(\theta_{\overline{\delta}}H_{\beta\overline{\tau}\delta}-\theta_{\delta}H_{\beta\overline{\tau}\overline{\delta}}+H_{\overline{\tau}\alpha\delta}H_{\beta\overline{\alpha}\overline{\delta}}\big)
    \nonumber\\&(\nabla_\alpha H)_{\beta\overline{\tau}\overline{\alpha}}=-\frac{1}{2}(\mathcal{L}_{\theta^\sharp}g)_{\beta\overline{\tau}}+\frac{1}{2}(\iota_{\nabla f}H)_{\beta\overline{\tau}}-\frac{1}{2}\big(\theta_{\overline{\delta}}H_{\beta\overline{\tau}\delta}-\theta_{\delta}H_{\beta\overline{\tau}\overline{\delta}}+H_{\overline{\tau}\alpha\delta}H_{\beta\overline{\alpha}\overline{\delta}}\big).\label{ll1}
\end{align}

\end{lemma}

\begin{proof}
Apply the Bianchi identity (\ref{10}) and (\ref{28}),
\begin{align*}
    R^B_{\beta \overline{\tau}\alpha\overline{\alpha}}+ R^B_{\overline{\tau}\alpha\beta\overline{\alpha}}+ R^B_{\alpha \beta \overline{\tau}\overline{\alpha}}&=(\nabla^B_{\overline{\alpha}} H)_{\beta\overline{\tau}\alpha} -\langle H(\partial_\beta,\partial_{\overline{\tau}}),H(\partial_\alpha,\partial_{\overline{\alpha}})  \rangle-\langle H(\partial_{\overline{\tau}},\partial_{\alpha}),H(\partial_\beta,\partial_{\overline{\alpha}})  \rangle-\langle H(\partial_\alpha,\partial_{\beta}),H(\partial_{\overline{\tau}},\partial_{\overline{\alpha}})  \rangle
    \\&=(\nabla^B_{\overline{\alpha}} H)_{\beta\overline{\tau}\alpha} -\theta_{\overline{\delta}}H_{\beta\overline{\tau}\delta}+\theta_{\delta}H_{\beta\overline{\tau}\overline{\delta}}-H_{\overline{\tau}\alpha\delta}H_{\beta\overline{\alpha}\overline{\delta}}-H_{\overline{\tau}\alpha\overline{\delta}}H_{\beta\overline{\alpha}\delta}-H_{\alpha\beta\overline{\delta}}H_{\overline{\tau}\overline{\alpha}\delta}.
\end{align*}    
Let $V=\frac{1}{2}(\theta^{\sharp}-\nabla f)$. By (\ref{V}), we get $\rho_B(\partial_\beta,J\partial_{\overline{\tau}})=(\mathcal{L}_Vg)_{\beta\overline{\tau}}$. Then, by (\ref{29}),
\begin{align*}
    (\nabla^B_{\overline{\alpha}} H)_{\beta\overline{\tau}\alpha} &= R^B_{\beta \overline{\tau}\alpha\overline{\alpha}}+ R^B_{\overline{\tau}\alpha\beta\overline{\alpha}}+\theta_{\overline{\delta}}H_{\beta\overline{\tau}\delta}-\theta_{\delta}H_{\beta\overline{\tau}\overline{\delta}}+H_{\overline{\tau}\alpha\delta}H_{\beta\overline{\alpha}\overline{\delta}}
    \\&=\rho_B(\partial_\beta,J\partial_{\overline{\tau}})-R^B_{\overline{\tau}\beta}+\theta_{\overline{\delta}}H_{\beta\overline{\tau}\delta}-\theta_{\delta}H_{\beta\overline{\tau}\overline{\delta}}+H_{\overline{\tau}\alpha\delta}H_{\beta\overline{\alpha}\overline{\delta}}
    \\&=(\mathcal{L}_Vg)_{\beta\overline{\tau}}+\nabla^B_{\overline{\tau}}\nabla^B_\beta f+\theta_{\overline{\delta}}H_{\beta\overline{\tau}\delta}-\theta_{\delta}H_{\beta\overline{\tau}\overline{\delta}}+H_{\overline{\tau}\alpha\delta}H_{\beta\overline{\alpha}\overline{\delta}}
    \\&=\frac{1}{2}(\mathcal{L}_{\theta^\sharp}g)_{\beta\overline{\tau}}+\frac{1}{2}(\iota_{\nabla f}H)_{\beta\overline{\tau}}+\theta_{\overline{\delta}}H_{\beta\overline{\tau}\delta}-\theta_{\delta}H_{\beta\overline{\tau}\overline{\delta}}+H_{\overline{\tau}\alpha\delta}H_{\beta\overline{\alpha}\overline{\delta}},
\end{align*}
where we used the soliton assumption that $R^B_{\overline{\tau}\beta}=-\nabla^B_{\overline{\tau}}\nabla^B_\beta f$. By the definition (\ref{20}),
\begin{align*}
    (\nabla^B_{\overline{\alpha}} H)_{\beta\overline{\tau}\alpha} &=(\nabla_{\overline{\alpha}} H)_{\beta\overline{\tau}\alpha}-\frac{1}{2}\langle H(\partial_{\overline{\alpha}},\partial_{\beta}),H(\partial_{\overline{\tau}},\partial_{\alpha})  \rangle+\frac{1}{2}\langle H(\partial_{\overline{\alpha}},\partial_{\overline{\tau}}),H(\partial_\beta,\partial_{\alpha})  \rangle-\frac{1}{2}\langle H(\partial_\beta,\partial_{\overline{\tau}}),H(\partial_{\overline{\alpha}},\partial_{\alpha})  \rangle
    \\&=(\nabla_{\overline{\alpha}} H)_{\beta\overline{\tau}\alpha}+\frac{1}{2}\big(\theta_{\overline{\delta}}H_{\beta\overline{\tau}\delta}-\theta_{\delta}H_{\beta\overline{\tau}\overline{\delta}}+H_{\overline{\tau}\alpha\delta}H_{\beta\overline{\alpha}\overline{\delta}}\big)
\end{align*}
Other parts follow similarly.

\end{proof}

In the following, let $\eta$ be an anti-Hermitian, symmetric 2-tensor, satisfying the formula (\ref{k2}). For convenience, we denote three real forms $H_1, H_2,H_3$ by 
\begin{align}
    H_1(\eta)=H_{\alpha\overline{\beta}\tau}H_{\overline{\alpha}\beta \overline{\delta}}\eta_{\overline{\tau}\overline{\gamma}}\eta_{\delta\gamma}, \quad  H_2(\eta)=H_{\alpha\beta\overline{\delta}}H_{\overline{\alpha}\overline{\beta}\tau}\eta_{\overline{\tau}\overline{\gamma}}\eta_{\delta\gamma},\quad H_3(\eta)=H_{\alpha\overline{\beta}\tau}H_{\overline{\alpha}\overline{\gamma}\delta}\eta_{\beta\gamma}\eta_{\overline{\tau}\overline{\delta}}.
\end{align}

Since $\eta$ satisfies the formula (\ref{k2}), we further derive the following proposition.
\begin{lemma}
 Suppose $\eta$ is an anti-Hermitian, symmetric 2-tensor, satisfying the formula (\ref{k2}). Then,
 \begin{align}
     H_2(\eta)=-2H_3(\eta), \quad (\nabla^B_\alpha\eta)_{\beta\gamma}H_{\overline{\alpha}\overline{\beta}\tau}\eta_{\overline{\gamma}\overline{\tau}}=-H_2(\eta). \label{ll}
 \end{align}
\end{lemma}
\begin{proof}
Recall the formula (\ref{k2}), 
\begin{align*}
    (\nabla^B_\alpha\eta)_{\beta\gamma}-(\nabla^B_\beta\eta)_{\alpha\gamma}=-H_{\alpha\beta\overline{\tau}}\eta_{\tau\gamma}+H_{\beta\gamma\overline{\tau}}\eta_{\tau\alpha}+H_{\gamma\alpha\overline{\tau}}\eta_{\tau\beta}.
\end{align*}
Since $\eta$ is symmetric, we deduce that
\begin{align}
    0=H_{\alpha\beta\overline{\tau}}\eta_{\tau\gamma}+H_{\beta\gamma\overline{\tau}}\eta_{\tau\alpha}+H_{\gamma\alpha\overline{\tau}}\eta_{\tau\beta}. \label{ab}
\end{align}
Therefore,
\begin{align*}
    H_2(\eta)&=H_{\alpha\beta\overline{\delta}}H_{\overline{\alpha}\overline{\beta}\tau}\eta_{\overline{\tau}\overline{\gamma}}\eta_{\delta\gamma}
    \\&=-(H_{\beta\gamma\overline{\delta}}\eta_{\delta\alpha}+H_{\gamma\alpha\overline{\delta}}\eta_{\delta\beta})H_{\overline{\alpha}\overline{\beta}\tau}\eta_{\overline{\tau}\overline{\gamma}}
    \\&=-2H_3(\eta).
\end{align*}
Moreover, we compute 
\begin{align*}
    (\nabla^B_\alpha\eta)_{\beta\gamma}H_{\overline{\alpha}\overline{\beta}\tau}\eta_{\overline{\gamma}\overline{\tau}}&=\big( (\nabla^B_\beta\eta)_{\alpha\gamma}-H_{\alpha\beta\overline{\delta}}\eta_{\delta\gamma}+H_{\beta\gamma\overline{\delta}}\eta_{\delta\alpha}+H_{\gamma\alpha\overline{\delta}}\eta_{\delta\beta} \big)H_{\overline{\alpha}\overline{\beta}\tau}\eta_{\overline{\gamma}\overline{\tau}}
    \\&=-(\nabla^B_\alpha\eta)_{\beta\gamma}H_{\overline{\alpha}\overline{\beta}\tau}\eta_{\overline{\gamma}\overline{\tau}}-H_2(\eta)+2H_3(\eta).
\end{align*}
Hence,
\begin{align*}
    (\nabla^B_\alpha\eta)_{\beta\gamma}H_{\overline{\alpha}\overline{\beta}\tau}\eta_{\overline{\gamma}\overline{\tau}}=-H_2(\eta).
\end{align*}
\end{proof}

Using the integration by part, we can deduce the following proposition.
\begin{proposition}
Let $(M^{2n},g,H,J,f)$ be a pluriclosed steady soliton and $\eta$ is an anti-Hermitian, symmetric 2-tensor, satisfying the equation (\ref{k2}). Then, 
\begin{align}
    \nonumber\int_M &\Big( (\nabla^B_{\overline{\alpha}}\eta)_{\beta\gamma}H_{\alpha\overline{\beta}\tau}\eta_{\overline{\gamma}\overline{\tau}}+(\nabla^B_{\alpha}\eta)_{\overline{\beta}\overline{\gamma}}H_{\overline{\alpha}\beta\overline{\tau}}\eta_{\gamma\tau}\Big)e^{-f}dV_g+\int_M \Big(\nabla_{\overline{\alpha}}f H_{\alpha\beta\overline{\tau}}\eta_{\overline{\beta}\overline{\gamma}}\eta_{\tau\gamma}-\nabla_{\alpha}f H_{\overline{\alpha}\beta\overline{\tau}}\eta_{\overline{\beta}\overline{\gamma}}\eta_{\tau\gamma} \Big)e^{-f}dV_g
    \\&=\int_M -2H_1(\eta) e^{-f}dV_g+\int_M (\mathcal{L}_{\theta^\sharp }g)_{\beta\overline{\tau}}\eta_{\overline{\beta}\overline{\gamma}}\eta_{\tau\gamma}e^{-f}dV_g, \label{pp1}
\end{align}
 and
\begin{align}
    \int_M \big(\theta_{\overline{\delta}}H_{\beta\overline{\tau}\delta}-\theta_{\delta}H_{\beta\overline{\tau}\overline{\delta}}\big) \eta_{\overline{\beta}\overline{\gamma}}\eta_{\tau\gamma} e^{-f}dV_g=\int_M -H_1(\eta)e^{-f}dV_g. \label{pp2}
\end{align}

\end{proposition}
\begin{proof}
 Using (\ref{ll1}), we derive that
 \begin{align*}
     \int_M&\big((\nabla^B_{\overline{\alpha}}H)_{\beta\overline{\tau}\alpha}-(\nabla^B_\alpha H)_{\beta\overline{\tau}\overline{\alpha}}\big) \eta_{\overline{\beta}\overline{\gamma}}\eta_{\tau\gamma} e^{-f}dV_g
     \\&=\int_M (\mathcal{L}_{\theta^\sharp }g)_{\beta\overline{\tau}}\eta_{\overline{\beta}\overline{\gamma}}\eta_{\tau\gamma}e^{-f}dV_g+\int_M 2\big(\theta_{\overline{\delta}}H_{\beta\overline{\tau}\delta}-\theta_{\delta}H_{\beta\overline{\tau}\overline{\delta}}+H_{\overline{\tau}\alpha\delta}H_{\beta\overline{\alpha}\overline{\delta}}\big) \eta_{\overline{\beta}\overline{\gamma}}\eta_{\tau\gamma} e^{-f}dV_g,
 \end{align*}
 and
 \begin{align*}
      \int_M&\big((\nabla_{\overline{\alpha}}H)_{\beta\overline{\tau}\alpha}-(\nabla_\alpha H)_{\beta\overline{\tau}\overline{\alpha}}\big) \eta_{\overline{\beta}\overline{\gamma}}\eta_{\tau\gamma} e^{-f}dV_g
     \\&=\int_M (\mathcal{L}_{\theta^\sharp }g)_{\beta\overline{\tau}}\eta_{\overline{\beta}\overline{\gamma}}\eta_{\tau\gamma}e^{-f}dV_g+\int_M \big(\theta_{\overline{\delta}}H_{\beta\overline{\tau}\delta}-\theta_{\delta}H_{\beta\overline{\tau}\overline{\delta}}+H_{\overline{\tau}\alpha\delta}H_{\beta\overline{\alpha}\overline{\delta}}\big) \eta_{\overline{\beta}\overline{\gamma}}\eta_{\tau\gamma} e^{-f}dV_g.
 \end{align*}
 Using integration by parts, we get 
\begin{align*}
    \int_M\big(&(\nabla^B_{\overline{\alpha}}H)_{\beta\overline{\tau}\alpha}-(\nabla^B_\alpha H)_{\beta\overline{\tau}\overline{\alpha}}\big) \eta_{\overline{\beta}\overline{\gamma}}\eta_{\tau\gamma} e^{-f}dV_g
    \\&=\int_M \Big( -(\nabla^B_\alpha\eta)_{\beta\gamma}H_{\overline{\alpha}\overline{\beta}\tau}\eta_{\overline{\gamma}\overline{\tau}}-(\nabla^B_{\overline{\alpha}}\eta)_{\overline{\beta}\overline{\gamma}}H_{\alpha\beta\overline{\tau}}\eta_{\gamma\tau}+ (\nabla^B_{\overline{\alpha}}\eta)_{\beta\gamma}H_{\alpha\overline{\beta}\tau}\eta_{\overline{\gamma}\overline{\tau}}+(\nabla^B_{\alpha}\eta)_{\overline{\beta}\overline{\gamma}}H_{\overline{\alpha}\beta\overline{\tau}}\eta_{\gamma\tau}\Big)e^{-f}dV_g
    \\&\kern2em+\int_M \Big(\nabla_{\overline{\alpha}}f H_{\alpha\beta\overline{\tau}}\eta_{\overline{\beta}\overline{\gamma}}\eta_{\tau\gamma}-\nabla_{\alpha}f H_{\overline{\alpha}\beta\overline{\tau}}\eta_{\overline{\beta}\overline{\gamma}}\eta_{\tau\gamma} \Big)e^{-f}dV_g
    \\&=\int_M \Big( (\nabla^B_{\overline{\alpha}}\eta)_{\beta\gamma}H_{\alpha\overline{\beta}\tau}\eta_{\overline{\gamma}\overline{\tau}}+(\nabla^B_{\alpha}\eta)_{\overline{\beta}\overline{\gamma}}H_{\overline{\alpha}\beta\overline{\tau}}\eta_{\gamma\tau}\Big)e^{-f}dV_g+\int_M \Big(\nabla_{\overline{\alpha}}f H_{\alpha\beta\overline{\tau}}\eta_{\overline{\beta}\overline{\gamma}}\eta_{\tau\gamma}-\nabla_{\alpha}f H_{\overline{\alpha}\beta\overline{\tau}}\eta_{\overline{\beta}\overline{\gamma}}\eta_{\tau\gamma} \Big)e^{-f}dV_g
    \\&\kern2em +2\int_M H_2(\eta) e^{-f}dV_g
\end{align*}
Since
\begin{align*}
    &(\nabla^B_{\overline{\alpha}}\eta)_{\beta\gamma}H_{\alpha\overline{\beta}\tau}\eta_{\overline{\gamma}\overline{\tau}}=(\nabla_{\overline{\alpha}}\eta)_{\beta\gamma}H_{\alpha\overline{\beta}\tau}\eta_{\overline{\gamma}\overline{\tau}}-\frac{1}{2}H_1(\eta)+\frac{1}{2}H_3(\eta),
    \\&(\nabla^B_{\alpha}\eta)_{\beta\gamma}H_{\overline{\alpha}\overline{\beta}\tau}\eta_{\overline{\gamma}\overline{\tau}}=(\nabla_{\alpha}\eta)_{\beta\gamma}H_{\overline{\alpha}\overline{\beta}\tau}\eta_{\overline{\gamma}\overline{\tau}}-\frac{1}{2}H_2(\eta)+\frac{1}{2}H_3(\eta),
\end{align*}
 we compute
 \begin{align*}
     \int_M\big(&(\nabla_{\overline{\alpha}}H)_{\beta\overline{\tau}\alpha}-(\nabla_\alpha H)_{\beta\overline{\tau}\overline{\alpha}}\big) \eta_{\overline{\beta}\overline{\gamma}}\eta_{\tau\gamma} e^{-f}dV_g
    \\&=\int_M \Big( -(\nabla_\alpha\eta)_{\beta\gamma}H_{\overline{\alpha}\overline{\beta}\tau}\eta_{\overline{\gamma}\overline{\tau}}-(\nabla_{\overline{\alpha}}\eta)_{\overline{\beta}\overline{\gamma}}H_{\alpha\beta\overline{\tau}}\eta_{\gamma\tau}+ (\nabla_{\overline{\alpha}}\eta)_{\beta\gamma}H_{\alpha\overline{\beta}\tau}\eta_{\overline{\gamma}\overline{\tau}}+(\nabla_{\alpha}\eta)_{\overline{\beta}\overline{\gamma}}H_{\overline{\alpha}\beta\overline{\tau}}\eta_{\gamma\tau}\Big)e^{-f}dV_g
    \\&\kern2em+\int_M \Big(\nabla_{\overline{\alpha}}f H_{\alpha\beta\overline{\tau}}\eta_{\overline{\beta}\overline{\gamma}}\eta_{\tau\gamma}-\nabla_{\alpha}f H_{\overline{\alpha}\beta\overline{\tau}}\eta_{\overline{\beta}\overline{\gamma}}\eta_{\tau\gamma} \Big)e^{-f}dV_g
    \\&=\int_M \Big( -(\nabla^B_\alpha\eta)_{\beta\gamma}H_{\overline{\alpha}\overline{\beta}\tau}\eta_{\overline{\gamma}\overline{\tau}}-(\nabla^B_{\overline{\alpha}}\eta)_{\overline{\beta}\overline{\gamma}}H_{\alpha\beta\overline{\tau}}\eta_{\gamma\tau}+ (\nabla^B_{\overline{\alpha}}\eta)_{\beta\gamma}H_{\alpha\overline{\beta}\tau}\eta_{\overline{\gamma}\overline{\tau}}+(\nabla^B_{\alpha}\eta)_{\overline{\beta}\overline{\gamma}}H_{\overline{\alpha}\beta\overline{\tau}}\eta_{\gamma\tau}\Big)e^{-f}dV_g
    \\&\kern2em+\int_M \Big(\nabla_{\overline{\alpha}}f H_{\alpha\beta\overline{\tau}}\eta_{\overline{\beta}\overline{\gamma}}\eta_{\tau\gamma}-\nabla_{\alpha}f H_{\overline{\alpha}\beta\overline{\tau}}\eta_{\overline{\beta}\overline{\gamma}}\eta_{\tau\gamma} \Big)e^{-f}dV_g+\int_M \big(-H_2(\eta)+H_1(\eta)\big) e^{-f}dV_g
    \\&=\int_M \Big( (\nabla^B_{\overline{\alpha}}\eta)_{\beta\gamma}H_{\alpha\overline{\beta}\tau}\eta_{\overline{\gamma}\overline{\tau}}+(\nabla^B_{\alpha}\eta)_{\overline{\beta}\overline{\gamma}}H_{\overline{\alpha}\beta\overline{\tau}}\eta_{\gamma\tau}\Big)e^{-f}dV_g+\int_M \Big(\nabla_{\overline{\alpha}}f H_{\alpha\beta\overline{\tau}}\eta_{\overline{\beta}\overline{\gamma}}\eta_{\tau\gamma}-\nabla_{\alpha}f H_{\overline{\alpha}\beta\overline{\tau}}\eta_{\overline{\beta}\overline{\gamma}}\eta_{\tau\gamma} \Big)e^{-f}dV_g
    \\&\kern2em +\int_M \big(H_2(\eta)+H_1(\eta)\big) e^{-f}dV_g
 \end{align*}
Therefore, we conclude that 
\begin{align*}
    \int_M &\Big( (\nabla^B_{\overline{\alpha}}\eta)_{\beta\gamma}H_{\alpha\overline{\beta}\tau}\eta_{\overline{\gamma}\overline{\tau}}+(\nabla^B_{\alpha}\eta)_{\overline{\beta}\overline{\gamma}}H_{\overline{\alpha}\beta\overline{\tau}}\eta_{\gamma\tau}\Big)e^{-f}dV_g+\int_M \Big(\nabla_{\overline{\alpha}}f H_{\alpha\beta\overline{\tau}}\eta_{\overline{\beta}\overline{\gamma}}\eta_{\tau\gamma}-\nabla_{\alpha}f H_{\overline{\alpha}\beta\overline{\tau}}\eta_{\overline{\beta}\overline{\gamma}}\eta_{\tau\gamma} \Big)e^{-f}dV_g
    \\&=\int_M -2H_1(\eta) e^{-f}dV_g+\int_M (\mathcal{L}_{\theta^\sharp }g)_{\beta\overline{\tau}}\eta_{\overline{\beta}\overline{\gamma}}\eta_{\tau\gamma}e^{-f}dV_g.
\end{align*}
 and
\begin{align*}
    \int_M \big(\theta_{\overline{\delta}}H_{\beta\overline{\tau}\delta}-\theta_{\delta}H_{\beta\overline{\tau}\overline{\delta}}\big) \eta_{\overline{\beta}\overline{\gamma}}\eta_{\tau\gamma} e^{-f}dV_g=\int_M -H_1(\eta)e^{-f}dV_g.
\end{align*} 
\end{proof}

Next, we consider the Bismut curvature. 
\begin{lemma}
Let $(M,g,H,J,f)$ be a pluriclosed steady soliton. Then,
\begin{align*}
    R^B_{\alpha\overline{\alpha}\beta\overline{\tau}}=-\nabla_\beta\nabla_{\overline{\tau}}f-\frac{1}{2}(\mathcal{L}_{\theta^\sharp }g)_{\beta\overline{\tau}}-\theta_{\overline{\delta}}H_{\beta\overline{\tau}\delta}+\theta_{\delta}H_{\beta\overline{\tau}\overline{\delta}}-H_{\overline{\tau}\alpha\delta}H_{\beta\overline{\alpha}\overline{\delta}}.
\end{align*} 
Suppose $\eta$ is an anti-Hermitian, symmetric 2-tensor, satisfying (\ref{k2}), then
\begin{align}
    \int_M R^B_{\alpha\overline{\alpha}\beta\overline{\tau}}\eta_{\overline{\beta}\overline{\gamma}}\eta_{\tau\gamma}e^{-f}dV_g=\int_M (-\nabla_\beta\nabla_{\overline{\tau}}f)\eta_{\overline{\beta}\overline{\gamma}}\eta_{\tau\gamma}e^{-f}dV_g-\frac{1}{2}\int_M (\mathcal{L}_{\theta^\sharp }g)_{\beta\overline{\tau}}\eta_{\overline{\beta}\overline{\gamma}}\eta_{\tau\gamma}e^{-f}dV_g+\int_M \big( H_1(\eta)-H_2(\eta) \big) e^{-f}dV_g. \label{pp3}
\end{align}
\end{lemma}
\begin{proof}
Using the Bianchi identity (\ref{10}),
\begin{align*}
    R^B_{\alpha\overline{\alpha}\beta\overline{\tau}}+R^B_{\overline{\alpha}\beta\alpha\overline{\tau}}+R^B_{\beta\alpha\overline{\alpha}\overline{\tau}}&=(\nabla^B_{\overline{\tau}} H)_{\alpha\overline{\alpha}\beta} -\langle H(\partial_\alpha,\partial_{\overline{\alpha}}),H(\partial_\beta,\partial_{\overline{\tau}})  \rangle-\langle H(\partial_{\overline{\alpha}},\partial_{\beta}),H(\partial_\alpha,\partial_{\overline{\tau}})  \rangle-\langle H(\partial_\beta,\partial_{\alpha}),H(\partial_{\overline{\alpha}},\partial_{\overline{\tau}})  \rangle
    \\&=-(\nabla^B_{\overline{\tau}}\theta)_\beta-\theta_{\overline{\delta}}H_{\beta\overline{\tau}\delta}+\theta_{\delta}H_{\beta\overline{\tau}\overline{\delta}}-H_{\overline{\tau}\alpha\delta}H_{\beta\overline{\alpha}\overline{\delta}}.
\end{align*}
Similarly, 
\begin{align*}
R^B_{\alpha\overline{\alpha}\overline{\tau}\beta}+R^B_{\overline{\alpha}\overline{\tau}\alpha\beta}+R^B_{\overline{\tau}\overline{\alpha}\alpha\beta}&=(\nabla^B_{\beta}\theta)_{\overline{\tau}}-\big(\theta_{\overline{\delta}}H_{\beta\overline{\tau}\delta}+\theta_{\delta}H_{\beta\overline{\tau}\overline{\delta}}-H_{\overline{\tau}\alpha\delta}H_{\beta\overline{\alpha}\overline{\delta}}\big).
\end{align*}
Thus,
\begin{align*}
    R^B_{\alpha\overline{\alpha}\beta\overline{\tau}}&=R^B_{\beta\overline{\tau}}-(\nabla^B_{\overline{\tau}}\theta)_\beta-\big(\theta_{\overline{\delta}}H_{\beta\overline{\tau}\delta}-\theta_{\delta}H_{\beta\overline{\tau}\overline{\delta}}+H_{\overline{\tau}\alpha\delta}H_{\beta\overline{\alpha}\overline{\delta}}\big)
    \\&=R^B_{\overline{\tau}\beta}-(\nabla^B_{\beta}\theta)_{\overline{\tau}}-\big(\theta_{\overline{\delta}}H_{\beta\overline{\tau}\delta}-\theta_{\delta}H_{\beta\overline{\tau}\overline{\delta}}+H_{\overline{\tau}\alpha\delta}H_{\beta\overline{\alpha}\overline{\delta}}\big)
    \\&=-\nabla_\beta\nabla_{\overline{\tau}}f-\frac{1}{2}(\mathcal{L}_{\theta^\sharp }g)_{\beta\overline{\tau}}-\theta_{\overline{\delta}}H_{\beta\overline{\tau}\delta}+\theta_{\delta}H_{\beta\overline{\tau}\overline{\delta}}-H_{\overline{\tau}\alpha\delta}H_{\beta\overline{\alpha}\overline{\delta}}
\end{align*}
where we used the formula (\ref{s}). The next result is followed by (\ref{pp2}).

\end{proof}

\begin{lemma}
Suppose $(M,g,H,J,f)$ is a pluriclosed steady soliton and $\eta$ is an anti-Hermitian, symmetric 2-tensor, $\divg^B_f\eta=0$, satisfying (\ref{k2}). Then,    
\begin{align}
    \nonumber\int_M \langle R^B(\eta),\eta \rangle e^{-f}dV_g&=\int_M 2(\nabla^B_\alpha\eta)_{\beta\gamma}(\nabla^B_{\overline{\alpha}}\eta)_{\overline{\beta}\overline{\gamma}} e^{-f}dV_g+\int_M \big(2H_1(\eta)-2H_2(\eta)\big)e^{-f}dV_g
    \\&\kern2em-\int_M (\mathcal{L}_{\theta^\sharp }g)_{\beta\overline{\tau}}\eta_{\overline{\beta}\overline{\gamma}}\eta_{\tau\gamma}e^{-f}dV_g+ \int_M \Big(\nabla_{\overline{\alpha}}f H_{\alpha\beta\overline{\tau}}\eta_{\overline{\beta}\overline{\gamma}}\eta_{\tau\gamma}-\nabla_{\alpha}f H_{\overline{\alpha}\beta\overline{\tau}}\eta_{\overline{\beta}\overline{\gamma}}\eta_{\tau\gamma} \Big)e^{-f}dV_g. \label{ll2}
\end{align}
\end{lemma}

\begin{proof}
Using the integration by part and the commutator formula (\ref{cf}), 
\begin{align*}
    \int_M (\nabla^B_\alpha\eta)_{\beta\gamma}(\nabla^B_{\overline{\beta}}\eta)_{\overline{\alpha}\overline{\gamma}} e^{-f}dV_g&=\int_M -(\nabla^B_{\overline{\beta}}\nabla^B_\alpha\eta)_{\beta\gamma}\eta_{\overline{\alpha}\overline{\gamma}} e^{-f}dV_g+\int_M (\nabla^B_\alpha\eta)_{\beta\gamma}\eta_{\overline{\alpha}\overline{\gamma}}\nabla_{\overline{\beta}}f e^{-f}dV_g
    \\&=\int_M -\Big( (\nabla^B_{\alpha}\nabla^B_{\overline{\beta}}\eta)_{\beta\gamma}-R^B_{\overline{\beta}\alpha\beta\overline{\tau}}\eta_{\tau\gamma}-R^B_{\overline{\beta}\alpha\gamma\overline{\tau}}\eta_{\beta\tau}-H_{\overline{\beta}\alpha\tau}(\nabla^B_{\overline{\tau}}\eta)_{\beta\gamma}-H_{\overline{\beta}\alpha\overline{\tau}}(\nabla^B_{\tau}\eta)_{\beta\gamma} \Big)\eta_{\overline{\alpha}\overline{\gamma}} e^{-f}dV_g
    \\&\kern2em -\int_M \nabla^B_\alpha\nabla^B_{\overline{\beta}}f \eta_{\overline{\alpha}\overline{\gamma}}\eta_{\beta\gamma} e^{-f}dV_g
    \\&=\int_M \Big( \mathring{R}^B(\eta)_{\alpha\gamma}+H_{\tau\overline{\beta}\alpha}(\nabla^B_{\overline{\tau}}\eta)_{\beta\gamma}+H_{\overline{\tau}\overline{\beta}\alpha}(\nabla^B_{\tau}\eta)_{\beta\gamma} \Big)\eta_{\overline{\alpha}\overline{\gamma}}e^{-f}dV_g. 
\end{align*}
By (\ref{ll}) and (\ref{pp1}), we have
\begin{align*}
    \int_M \langle R^B(\eta),\eta \rangle e^{-f}dV_g&=\int_M 2(\nabla^B_\alpha\eta)_{\beta\gamma}(\nabla^B_{\overline{\beta}}\eta)_{\overline{\alpha}\overline{\gamma}} e^{-f}dV_g
    \\&\kern2em -\int_M\Big(H_{\tau\overline{\beta}\alpha}(\nabla^B_{\overline{\tau}}\eta)_{\beta\gamma}\eta_{\overline{\alpha}\overline{\gamma}}+H_{\overline{\tau}\overline{\beta}\alpha}(\nabla^B_{\tau}\eta)_{\beta\gamma}\eta_{\overline{\alpha}\overline{\gamma}}+H_{\tau\beta\overline{\alpha}}(\nabla^B_{\overline{\tau}}\eta)_{\overline{\beta}\overline{\gamma}}\eta_{\alpha\gamma}+H_{\overline{\tau}\beta\overline{\alpha}}(\nabla^B_{\tau}\eta)_{\overline{\beta}\overline{\gamma}}\eta_{\alpha\gamma}   \Big) e^{-f}dV_g
    \\&=\int_M 2(\nabla^B_\alpha\eta)_{\beta\gamma}(\nabla^B_{\overline{\beta}}\eta)_{\overline{\alpha}\overline{\gamma}} e^{-f}dV_g+\int_M \big(2H_1(\eta)+2H_2(\eta) \big) e^{-f}dV_g-\int_M (\mathcal{L}_{\theta^\sharp }g)_{\beta\overline{\tau}}\eta_{\overline{\beta}\overline{\gamma}}\eta_{\tau\gamma}e^{-f}dV_g
    \\&\kern2em+ \int_M \Big(\nabla_{\overline{\alpha}}f H_{\alpha\beta\overline{\tau}}\eta_{\overline{\beta}\overline{\gamma}}\eta_{\tau\gamma}-\nabla_{\alpha}f H_{\overline{\alpha}\beta\overline{\tau}}\eta_{\overline{\beta}\overline{\gamma}}\eta_{\tau\gamma} \Big)e^{-f}dV_g.
\end{align*}
Note that 
\begin{align*}
    \int_M (\nabla^B_\alpha\eta)_{\beta\gamma}(\nabla^B_{\overline{\beta}}\eta)_{\overline{\alpha}\overline{\gamma}}e^{-f}dV_g&=\int_M (\nabla^B_\alpha\eta)_{\beta\gamma}\Big((\nabla^B_{\overline{\alpha}}\eta)_{\overline{\beta}\overline{\gamma}}+H_{\overline{\alpha}\overline{\beta}\tau}\eta_{\overline{\tau}\overline{\gamma}}-H_{\overline{\beta}\overline{\gamma}\tau}\eta_{\overline{\tau}\overline{\alpha}}-H_{\overline{\gamma}\overline{\alpha}\tau}\eta_{\overline{\tau}\overline{\beta}}\Big) e^{-f}dV_g
    \\&=\int_M (\nabla^B_\alpha\eta)_{\beta\gamma}(\nabla^B_{\overline{\alpha}}\eta)_{\overline{\beta}\overline{\gamma}}e^{-f}dV_g+\int_M 2(\nabla^B_\alpha\eta)_{\beta\gamma}H_{\overline{\alpha}\overline{\beta}\tau}\eta_{\overline{\tau}\overline{\gamma}}e^{-f}dV_g
    \\&=\int_M (\nabla^B_\alpha\eta)_{\beta\gamma}(\nabla^B_{\overline{\alpha}}\eta)_{\overline{\beta}\overline{\gamma}}e^{-f}dV_g-\int_M 2H_2(\eta) e^{-f}dV_g
\end{align*}
so we complete the proof.
\end{proof}

Now, we are ready to prove the main result.
\begin{theorem}
Suppose $(M,g,H,J,f)$ is a pluriclosed steady soliton and $\eta$ is an anti-Hermitian, symmetric 2-tensor, $\divg^B_f\eta=0$ satisfying (\ref{k2}). Then,
\begin{align}
\int_M\langle \eta, \overline{\mathcal{L}}_f(\eta)   \rangle   e^{-f}dV_g&=-\frac{1}{6}\|C\|_f^2 +\int_M |\eta|^2\big( -\frac{1}{2}S_B+\mathcal{L}_Vf \big)e^{-f}dV_g,
\end{align}
where $\|\cdot\|_f$ denotes the norm of $f$-twisted $L^2$ inner product (\ref{6}), $C$ is defined in (\ref{C}), $S_B=\tr_\omega \rho_B$ denotes the Bismut scalar curvature and the vector field $V=\frac{1}{2}(\theta^\sharp-\nabla f)$.
\end{theorem}

\begin{proof}
We notice that 
\begin{align*}
    \int_M\langle \eta, \overline{\mathcal{L}}_f(\eta)   \rangle   e^{-f}dV_g&=\int_M \Big( -\frac{1}{2}|\overline{\nabla}\eta|^2+\langle \mathring{R}^B(\eta),\eta \rangle \Big) e^{-f}dV_g.
\end{align*}
We use (\ref{MBC}) to compute
\begin{align}
    \nonumber|\overline{\nabla}\eta|^2&=2(\overline{\nabla}_\alpha\eta)_{\beta\gamma}(\overline{\nabla}_{\overline{\alpha}}\eta)_{\overline{\beta}\overline{\gamma}}+2(\overline{\nabla}_\alpha\eta)_{\beta\overline{\gamma}}(\overline{\nabla}_{\overline{\alpha}}\eta)_{\overline{\beta}\gamma}+2(\overline{\nabla}_\alpha\eta)_{\overline{\beta}\gamma}(\overline{\nabla}_{\overline{\alpha}}\eta)_{\beta\overline{\gamma}}+2(\overline{\nabla}_\alpha\eta)_{\overline{\beta}\overline{\gamma}}(\overline{\nabla}_{\overline{\alpha}}\eta)_{\beta\gamma}
    \nonumber\\&=2\big((\nabla^B_\alpha\eta)_{\beta\gamma}+H_{\alpha\beta\overline{\tau}}\eta_{\tau\gamma}\big)\big((\nabla^B_{\overline{\alpha}}\eta)_{\overline{\beta}\overline{\gamma}}+H_{\overline{\alpha}\overline{\beta}\delta}\eta_{\overline{\delta}\overline{\gamma}} \big)+2(H_{\alpha\overline{\beta}\overline{\tau}}\eta_{\tau\gamma})(H_{\overline{\alpha}\beta\delta}\eta_{\overline{\delta}\overline{\gamma}})
    \nonumber\\&\kern2em+2\big( (\nabla^B_\alpha\eta)_{\overline{\beta}\overline{\gamma}}+H_{\alpha\overline{\beta}\tau}\eta_{\overline{\tau}\overline{\gamma}}\big)\big((\nabla^B_{\overline{\alpha}}\eta)_{\beta\gamma}+H_{\overline{\alpha}\beta\overline{\delta}}\eta_{\delta\gamma}\big)
    \nonumber\\&=2\big((\nabla^B_\alpha\eta)_{\beta\gamma}(\nabla^B_{\overline{\alpha}}\eta)_{\overline{\beta}\overline{\gamma}}+(\nabla^B_\alpha\eta)_{\overline{\beta}\overline{\gamma}}(\nabla^B_{\overline{\alpha}}\eta)_{\beta\gamma} \big)+2(\nabla^B_\alpha\eta)_{\beta\gamma} H_{\overline{\alpha}\overline{\beta}\delta}\eta_{\overline{\delta}\overline{\gamma}} +2(\nabla^B_{\overline{\alpha}}\eta)_{\overline{\beta}\overline{\gamma}}H_{\alpha\beta\overline{\tau}}\eta_{\tau\gamma}
    \nonumber\\&\kern2em +2(\nabla^B_\alpha\eta)_{\overline{\beta}\overline{\gamma}}H_{\overline{\alpha}\beta\overline{\delta}}\eta_{\delta\gamma}+2(\nabla^B_{\overline{\alpha}}\eta)_{\beta\gamma} H_{\alpha\overline{\beta}\tau}\eta_{\overline{\tau}\overline{\gamma}}+2H_2(\eta)+4H_1(\eta), \label{lll}
\end{align}
since $\eta$ is anti-Hermitian and $H\in \Lambda^{2,1}\oplus \Lambda^{1,2}$. Next, we derive 
\begin{align*}
     \int_M  &(\nabla^B_{\overline{\alpha}}\eta)_{\beta\gamma}(\nabla^B_{\alpha}\eta)_{\overline{\beta}\overline{\gamma}}e^{-f}dV_g
    \\&=\int_M-(\nabla^B_\alpha \nabla^B_{\overline{\alpha}}\eta)_{\beta\gamma}\eta_{\overline{\beta}\overline{\gamma}}e^{-f}dV_g+\int_M(\nabla^B_{\overline{\alpha}}\eta)_{\beta\gamma}\eta_{\overline{\beta}\overline{\gamma}}\nabla_{\alpha}fe^{-f}dV_g
    \\&=\int_M -\Big((\nabla^B_{\overline{\alpha}}\nabla^B_\alpha \eta)_{\beta\gamma}-R^B_{\alpha\overline{\alpha}\beta\overline{\tau}}\eta_{\tau\gamma}-R^B_{\alpha\overline{\alpha}\gamma\overline{\tau}}\eta_{\beta\tau}-H_{\alpha\overline{\alpha}\tau}(\nabla^B_{\overline{\tau}}\eta)_{\beta\gamma}- H_{\alpha\overline{\alpha}\overline{\tau}}(\nabla^B_{\tau}\eta)_{\beta\gamma}\Big)\eta_{\overline{\beta}\overline{\gamma}} e^{-f}dV_g
    \\&\kern2em +\int_M(\nabla^B_{\overline{\alpha}}\eta)_{\beta\gamma}\eta_{\overline{\beta}\overline{\gamma}}\nabla_{\alpha}fe^{-f}dV_g.
\end{align*}
Using (\ref{28}),
\begin{align*}
    \int_M  &(\nabla^B_{\overline{\alpha}}\eta)_{\beta\gamma}(\nabla^B_{\alpha}\eta)_{\overline{\beta}\overline{\gamma}}e^{-f}dV_g
    \\&=\int_M \Big( (\nabla^B_{\alpha}\eta)_{\beta\gamma}(\nabla^B_{\overline{\alpha}}\eta)_{\overline{\beta}\overline{\gamma}}+2R^B_{\alpha\overline{\alpha}\beta\overline{\tau}}\eta_{\tau\gamma}\eta_{\overline{\beta}\overline{\gamma}}-\theta_\tau (\nabla^B_{\overline{\tau}}\eta)_{\beta\gamma}\eta_{\overline{\beta}\overline{\gamma}}+\theta_{\overline{\tau}} (\nabla^B_{\tau}\eta)_{\beta\gamma}\eta_{\overline{\beta}\overline{\gamma}}\Big)e^{-f}dV_g
    \\&\kern2em +\int_M(\nabla^B_{\overline{\alpha}}\eta)_{\beta\gamma}\eta_{\overline{\beta}\overline{\gamma}}\nabla_{\alpha}fe^{-f}dV_g-\int_M(\nabla^B_{\alpha}\eta)_{\beta\gamma}\eta_{\overline{\beta}\overline{\gamma}}\nabla_{\overline{\alpha}}f e^{-f}dV_g
    \\&=\int_M \Big( (\nabla^B_{\alpha}\eta)_{\beta\gamma}(\nabla^B_{\overline{\alpha}}\eta)_{\overline{\beta}\overline{\gamma}}+2R^B_{\alpha\overline{\alpha}\beta\overline{\tau}}\eta_{\tau\gamma}\eta_{\overline{\beta}\overline{\gamma}}+\theta_\tau (\nabla^B_{\overline{\tau}}\eta)_{\overline{\beta}\overline{\gamma}}\eta_{\beta\gamma}+\theta_{\overline{\tau}} (\nabla^B_{\tau}\eta)_{\beta\gamma}\eta_{\overline{\beta}\overline{\gamma}}\Big)e^{-f}dV_g
    \\&\kern2em -\int_M(\nabla^B_{\overline{\alpha}}\eta)_{\overline{\beta}\overline{\gamma}} \eta_{\beta\gamma}\nabla_{\alpha}fe^{-f}dV_g-\int_M(\nabla^B_{\alpha}\eta)_{\beta\gamma}\eta_{\overline{\beta}\overline{\gamma}}\nabla_{\overline{\alpha}}f e^{-f}dV_g
    \\&\kern2em +\int_M |\eta|^2(-\nabla^B_{\overline{\alpha}}\nabla^B_\alpha f+(\nabla^B_{\overline{\tau}}\theta)_\tau+\frac{1}{2}|\nabla f|^2-\theta_\tau \nabla_{\overline{\tau}}f)e^{-f}dV_g.
\end{align*}
Using (\ref{k2}) again,
\begin{align*}
   \int_M(\nabla^B_{\alpha}\eta)_{\beta\gamma}\eta_{\overline{\beta}\overline{\gamma}}\nabla_{\overline{\alpha}}f e^{-f}dV_g&=\int_M \big( (\nabla^B_{\beta}\eta)_{\alpha\gamma}-H_{\alpha\beta\overline{\tau}}\eta_{\tau\gamma}+H_{\beta\gamma\overline{\tau}}\eta_{\tau\alpha}+H_{\gamma\alpha\overline{\tau}}\eta_{\tau\beta} \big)\eta_{\overline{\beta}\overline{\gamma}}\nabla_{\overline{\alpha}}f e^{-f}dV_g
   \\&=\int_M \big( -\nabla^B_\beta\nabla^B_{\overline{\alpha}}f\eta_{\alpha\gamma}-2H_{\alpha\beta\overline{\tau}}\eta_{\tau\gamma}\nabla_{\overline{\alpha}}f \big)\eta_{\overline{\beta}\overline{\gamma}} e^{-f}dV_g,
\end{align*}
\begin{align*}
    \int_M \theta_{\overline{\tau}} (\nabla^B_{\tau}\eta)_{\beta\gamma}\eta_{\overline{\beta}\overline{\gamma}} e^{-f}dV_g&=\int_M \theta_{\overline{\tau}}\big( (\nabla^B_\beta\eta)_{\tau\gamma}-H_{\tau\beta\overline{\delta}}\eta_{\delta\gamma}+H_{\beta\gamma\overline{\delta}}\eta_{\delta\tau}+H_{\gamma\tau\overline{\delta}}\eta_{\delta\beta}  \big)\eta_{\overline{\beta}\overline{\gamma}} e^{-f}dV_g
    \\&=\int_M \Big( -(\nabla^B_{\beta}\theta)_{\overline{\tau}}\eta_{\tau\gamma}\eta_{\overline{\beta}\overline{\gamma}}-2\theta_{\overline{\tau}}H_{\tau\beta\overline{\delta}}\eta_{\delta\gamma}\eta_{\overline{\beta}\overline{\gamma}} \Big) e^{-f}dV_g.
\end{align*}
Then,
\begin{align*}
    \int_M&  (\nabla^B_{\overline{\alpha}}\eta)_{\beta\gamma}(\nabla^B_{\alpha}\eta)_{\overline{\beta}\overline{\gamma}}e^{-f}dV_g
    \\&=\int_M \Big( (\nabla^B_{\alpha}\eta)_{\beta\gamma}(\nabla^B_{\overline{\alpha}}\eta)_{\overline{\beta}\overline{\gamma}}+2R^B_{\alpha\overline{\alpha}\beta\overline{\tau}}\eta_{\tau\gamma}\eta_{\overline{\beta}\overline{\gamma}}-2\theta_{\overline{\tau}}H_{\tau\beta\overline{\delta}}\eta_{\delta\gamma}\eta_{\overline{\beta}\overline{\gamma}}+2\theta_{\tau}H_{\overline{\tau}\beta\overline{\delta}}\eta_{\delta\gamma}\eta_{\overline{\beta}\overline{\gamma}}\Big)e^{-f}dV_g-\int_M (\mathcal{L}_{\theta^\sharp }g)_{\beta\overline{\tau}}\eta_{\overline{\beta}\overline{\gamma}}\eta_{\tau\gamma}e^{-f}dV_g
    \\&\kern2em +\int_M 2\nabla^B_\beta\nabla^B_{\overline{\alpha}}f\eta_{\alpha\gamma}\eta_{\overline{\beta}\overline{\gamma}}e^{-f}dV_g+\int_M 2\Big(\nabla_{\overline{\alpha}}f H_{\alpha\beta\overline{\tau}}\eta_{\overline{\beta}\overline{\gamma}}\eta_{\tau\gamma}-\nabla_{\alpha}f H_{\overline{\alpha}\beta\overline{\tau}}\eta_{\overline{\beta}\overline{\gamma}}\eta_{\tau\gamma} \Big)e^{-f}dV_g
    \\&\kern2em +\int_M \frac{1}{2}|\eta|^2\big(-d^*\theta-\triangle f-\langle \nabla f, \theta^\sharp-\nabla f \rangle\big)e^{-f}dV_g,
\end{align*}
where we use the fact that 
\begin{align*}
    \int_M \big( -\nabla^B_{\overline{\alpha}}\nabla^B_\alpha f+(\nabla^B_{\overline{\tau}}\theta)_\tau\big)|\eta|^2 e^{-f}dV_g=\int_M \frac{1}{2}(-d^*\theta-\triangle f)|\eta|^2 e^{-f}dV_g.
\end{align*}
Note that the integration is real so 
\begin{align*}
    \int_M (\iota_{\nabla f}H)_{\beta\overline{\tau}}\eta_{\overline{\beta}\overline{\gamma}}\eta_{\tau\gamma}e^{-f}dV_g=0.
\end{align*}
By (\ref{pp2}) and (\ref{pp3}), we have
\begin{align*}
    \int_M&  (\nabla^B_{\overline{\alpha}}\eta)_{\beta\gamma}(\nabla^B_{\alpha}\eta)_{\overline{\beta}\overline{\gamma}}e^{-f}dV_g
    \\&=\int_M \Big( (\nabla^B_{\alpha}\eta)_{\beta\gamma}(\nabla^B_{\overline{\alpha}}\eta)_{\overline{\beta}\overline{\gamma}}\Big)e^{-f}dV_g+\int_M \big( 4H_1(\eta)-2H_2(\eta) \big) e^{-f}dV_g
    \\&\kern2em -2\int_M (\mathcal{L}_{\theta^\sharp }g)_{\beta\overline{\tau}}\eta_{\overline{\beta}\overline{\gamma}}\eta_{\tau\gamma}e^{-f}dV_g+\int_M 2\Big(\nabla_{\overline{\alpha}}f H_{\alpha\beta\overline{\tau}}\eta_{\overline{\beta}\overline{\gamma}}\eta_{\tau\gamma}-\nabla_{\alpha}f H_{\overline{\alpha}\beta\overline{\tau}}\eta_{\overline{\beta}\overline{\gamma}}\eta_{\tau\gamma} \Big)e^{-f}dV_g
    \\&\kern2em +\int_M \frac{1}{2}|\eta|^2\big(-d^*\theta-\triangle f-\langle \nabla f, \theta^\sharp-\nabla f \rangle\big)e^{-f}dV_g.
\end{align*}
In summary, 
\begin{align*}
    \int_M |\overline{\nabla}\eta|^2 e^{-f}dV_g&=\int_M 2\big((\nabla^B_\alpha\eta)_{\beta\gamma}(\nabla^B_{\overline{\alpha}}\eta)_{\overline{\beta}\overline{\gamma}}+(\nabla^B_\alpha\eta)_{\overline{\beta}\overline{\gamma}}(\nabla^B_{\overline{\alpha}}\eta)_{\beta\gamma} \big) e^{-f}dV_g
    \\&\kern2em +\int_M 2(\nabla^B_\alpha\eta)_{\beta\gamma} H_{\overline{\alpha}\overline{\beta}\delta}\eta_{\overline{\delta}\overline{\gamma}}e^{-f}dV_g +\int_M 2(\nabla^B_{\overline{\alpha}}\eta)_{\overline{\beta}\overline{\gamma}}H_{\alpha\beta\overline{\tau}}\eta_{\tau\gamma} e^{-f}dV_g
    \nonumber\\&\kern2em +\int_M 2(\nabla^B_\alpha\eta)_{\overline{\beta}\overline{\gamma}}H_{\overline{\alpha}\beta\overline{\delta}}\eta_{\delta\gamma} e^{-f}dV_g+\int_M 2(\nabla^B_{\overline{\alpha}}\eta)_{\beta\gamma} H_{\alpha\overline{\beta}\tau}\eta_{\overline{\tau}\overline{\gamma}}e^{-f}dV_g+ \int_M \big(2H_2(\eta)+4H_1(\eta)\big) e^{-f}dV_g
    \\&=\int_M 4(\nabla^B_\alpha\eta)_{\beta\gamma}(\nabla^B_{\overline{\alpha}}\eta)_{\overline{\beta}\overline{\gamma}} e^{-f}dV_g-2\int_M (\mathcal{L}_{\theta^\sharp }g)_{\beta\overline{\tau}}\eta_{\overline{\beta}\overline{\gamma}}\eta_{\tau\gamma}e^{-f}dV_g+\int_M \big( 8H_1(\eta)-6H_2(\eta) \big) e^{-f}dV_g
    \\&\kern2em +2\int_M \Big(\nabla_{\overline{\alpha}}f H_{\alpha\beta\overline{\tau}}\eta_{\overline{\beta}\overline{\gamma}}\eta_{\tau\gamma}-\nabla_{\alpha}f H_{\overline{\alpha}\beta\overline{\tau}}\eta_{\overline{\beta}\overline{\gamma}}\eta_{\tau\gamma} \Big)e^{-f}dV_g 
    \\&\kern2em +\int_M |\eta|^2\big(-d^*\theta-\triangle f-\langle \nabla f, \theta^\sharp-\nabla f \rangle\big)e^{-f}dV_g
    \\&=2\int_M \langle R^B(\eta),\eta \rangle e^{-f}dV_g+\int_M \big( 4H_1(\eta)-2H_2(\eta) \big) e^{-f}dV_g+\int_M |\eta|^2\big(-d^*\theta-\triangle f-\langle \nabla f, \theta^\sharp-\nabla f \rangle\big)e^{-f}dV_g,
\end{align*}
where we used (\ref{ll2}). Finally, we derive
\begin{align*}
    \int_M\langle \eta, \overline{\mathcal{L}}_f(\eta)   \rangle   e^{-f}dV_g&=\int_M \big(-2H_1(\eta)+H_2(\eta)\big) e^{-f}dV_g
    \\&\kern2em +\int_M \frac{1}{2}|\eta|^2\big(\triangle f+d^*\theta+\langle \nabla f,\theta^\sharp-\nabla f \rangle\big)e^{-f}dV_g
\end{align*}
Recall that $C$ is defined by (\ref{C}), we have
\begin{align*}
    |C|^2&=3H_{km}^2(\eta\circ J)_{kl}(\eta\circ J)_{ml}+6H_{iab}H_{jcb}(\eta\circ J)_{ij}(\eta\circ J)_{ac}
    \\&=3(4H_1(\eta)+2H_2(\eta))+24H_3(\eta)
    \\&=12H_1(\eta)-6H_2(\eta)
\end{align*}
since $\eta$ is anti-Hermitian. It follows that 
\begin{align*}
    \int_M \big(-2H_1(\eta)+H_2(\eta)\big) e^{-f}dV_g=-\frac{1}{6}\|C\|_f^2.
\end{align*}
Taking the trace of (\ref{NT}), we note that the Bismut scalar curvature $S_B$ is given by 
\begin{align*}
    S_B=-d^*\theta-\triangle f.
\end{align*}
Therefore, our final result follows.

\end{proof}

In summary, we conclude that 
\begin{theorem}
Let $(M,g,H,J,f)$ be a compact pluriclosed steady soliton and $\gamma\in\ker\overline{\divg}_f$. Then, 
\begin{align}
     \frac{d^2}{dt^2}\Big|_{t=0}\lambda(\gamma)&= -2\|d^*_f\xi\|_f^2-\frac{1}{6}\|d\xi-C\|_f^2 +\int_M |\eta|^2\big( -\frac{1}{2}S_B+\mathcal{L}_Vf \big)e^{-f}dV_g, \label{k5}
\end{align} 
where $\gamma=\xi\circ J+\eta$ is the decomposition given in (\ref{k3}), $C$ is defined in (\ref{C}), $\|\cdot\|_f$ denotes the norm of $f$-twisted $L^2$ inner product (\ref{6}),  $S_B=\tr_\omega \rho_B$ denotes the Bismut scalar curvature and the vector field $V=\frac{1}{2}(\theta^\sharp-\nabla f)$.

\end{theorem}

\begin{corollary}
Let $(M,g,H,J)$ be a compact Bismut--Hermitian--Einstein manifold, then $(M,g,H,J)$ is linearly stable. Moreover, $\gamma=\xi\circ J+\eta$ lies in the kernel of second variation if and only if 
\begin{align*}
    d^*\xi=0,\quad d\xi=C.
\end{align*}
\end{corollary}
\begin{proof}
By \Cref{LV}, we notice that $\mathcal{L}_Vf=0$ when $g$ is a Bismut--Hermitian--Einstein metric. In this case, 
\begin{align*}
    \frac{d^2}{dt^2}\Big|_{t=0}\lambda(\gamma)&= -2\|d^*\xi\|^2-\frac{1}{6}\|d\xi-C\|^2 
\end{align*}
for any general variation $\gamma\in\ker\overline{\divg}$.
\end{proof}

\subsection{Infinitesimal deformation}

In the last part, we study the kernel of the second variation on Bismut-flat manifolds. By \cite{KK} Proposition 5.1, we define
\begin{defn}\label{IS}
Let $(M,g,H,f)$ be a steady gradient generalized Ricci soliton. A 2-tensor $\gamma\in \otimes^2 T^*M$ is called an \emph{(essential) infinitesimal generalized solitonic deformation} of $(M,g,H,f)$, if the operator $\overline{\mathcal{L}_f}(\gamma)=\frac{1}{2}\overline{\triangle_f}\gamma+R^B(\gamma)=0$ and $\gamma\in\ker\overline{\divg}$. In particular, in the Bismut-flar case, a 2-tensor $\gamma\in\ker\overline{\divg}$ is an infinitesimal generalized solitonic deformation if and only if  
\begin{align*}
    \overline{\nabla}\gamma=0.
\end{align*}
\end{defn}

From the definition (\ref{MBC}) and (\ref{20}), we see that
\begin{align*}
     (\overline{\nabla}_X\gamma)(Y,Z)=(\nabla^B_X\gamma)(Y,Z)+\gamma(H(X,Y),Z) \text{ 
 for all vector fields $X,Y,Z$.}
\end{align*}
Therefore, $\overline{\nabla}(\gamma\circ J)=0$ if $\overline{\nabla}\gamma=0$. Here, $\gamma\circ J=-\xi+\eta\circ J$. Apply Proposition 5.1, Lemma 5.5 and Lemma 5.6 in \cite{KK}, we have the following proposition.
\begin{proposition}
 Let $(M,g,H,J)$ be a compact Bismut-flat, pluriclosed manifold. Let $\gamma=\xi\circ J+\eta$ be an essential infinitesumal generalized solitonic deformation then
 \begin{align}
            \nabla_m(\eta\circ J)_{ij}=-\frac{1}{2}(H_{mik}\xi_{jk}+H_{mjk}\xi_{ik}),\quad\nabla_m\xi_{ij}=-\frac{1}{2}H_{mjk}(\eta\circ J)_{ik}+\frac{1}{2}H_{mik}(\eta\circ J)_{jk} \label{long}
        \end{align}
and
\begin{align}
      \big(\mathring{R}(\eta\circ J),\eta\circ J\big)_{L^2}=\big(\mathring{R}(\xi),\xi\big)_{L^2}=0,\quad \|\nabla (\eta\circ J)\|_{L^2}= \|\nabla \xi\|_{L^2},\quad \int_M R_{ij}(\eta\circ J)_{jk}(\eta\circ J)_{ik} dV_g= \int_M R_{ij}\xi_{jk}\xi_{ik} 
 dV_g.  \label{llong}
\end{align}
Besides,
\begin{align*}
    H_{mik}\xi_{ik}=0.
\end{align*}
\end{proposition}

In the following, let us consider a special case. Consider an essential infinitesimal generalized solitonic deformation $\gamma=\xi\circ J$ on a Bismut-flat, pluriclosed manifold $(M,g,J)$, i.e., the complex structure is fixed, (\ref{long}) and (\ref{llong}) imply that $\nabla\xi=0$ and
\begin{align*}
    0=H_{mik}\xi_{jk}+H_{mjk}\xi_{ik}, \quad \int_M R_{ij}\xi_{jk}\xi_{ik} 
 dV_g=0.
\end{align*}
Note that any Bismut-flat metric has nonnegative Ricci curvature,  we get the following corollary.

\begin{corollary}
 Let $(M,g,H,J)$ be a compact Bismut-flat, pluriclosed manifold. Suppose that the Ricci curvature is $(2n-1)$-positive, then there is no essential infinitesimal generalized solitonic deformation on directions of fixing the complex structure.    
\end{corollary}

\bibliographystyle{plain}
\bibliography{Reference}
\nocite{*}

\end{document}